\documentclass[12pt]{amsart}
\usepackage[alphabetic]{amsrefs}
\usepackage{mathdots, blkarray}
\usepackage{multicol}
\usepackage{graphicx}
\usepackage{caption}
\usepackage{amscd}
\usepackage{upgreek}
\usepackage{stmaryrd}
\usepackage{longtable}
\usepackage[T1]{fontenc}
\usepackage{latexsym, amsmath, amssymb, amsthm}
\usepackage[svgnames]{xcolor}
\usepackage{rsfso}
\usepackage[mathscr]{eucal}
\usepackage{mathtools}
\usepackage{mathptmx}
\usepackage{titletoc}
\usepackage{wrapfig}
\usepackage{float}
\usepackage{xypic}
\usepackage{microtype}
\usepackage{dsfont}
\usepackage{xcolor}
\usepackage{color}
\usepackage[colorlinks = true,
            linkcolor  = DarkBlue,
            urlcolor   = DarkRed,
            citecolor  = DarkGreen]{hyperref}
\usepackage{hyperref}
\allowdisplaybreaks	
\usepackage[centering, includeheadfoot, hmargin=1.0in, tmargin=0.5in,
  bmargin=0.6in, headheight=6pt]{geometry}
\newtheorem{theorem}{Theorem}[section]
\newtheorem{prop}[theorem]{Proposition}
\newtheorem{defn}[theorem]{\rm\textsc{Definition}}
\newtheorem{lem}[theorem]{Lemma}
\newtheorem{coro}[theorem]{Corollary}

\newtheorem{thm}[theorem]{Theorem}

\newtheorem{rem}[theorem]{\rm\textsc{Remark}}
\newtheorem{exam}[theorem]{\rm\textsc{Example}}

\newcommand{\bslash}{\kern-0.1em\texttt{\scalebox{0.6}[1]{/}}\kern-0.15em \texttt{\scalebox{0.6}[1]{/}}}

\begin{document}
\title[Strong quasi-MV* algebras and their logics]{Strong quasi-MV* algebras and their logics}

\author{Lei Cai}
\address{School of Mathematical Sciences, University of Jinan, No. 336, West Road of Nan Xinzhuang,
         Jinan, Shandong, 250022 P.R. China.}
\email{cailei@stu.ujn.edu.cn}

\author{Wenjuan Chen}
\address{School of Mathematical Sciences, University of Jinan, No. 336, West Road of Nan Xinzhuang,
         Jinan, Shandong, 250022 P.R. China.}
\email{wjchenmath@gmail.com}

\begin{abstract}
In this paper, we introduce the subvariety of quasi-MV* algebras in order to characterize the logic which is related to complex fuzzy logic. First, we give the definitions of strong quasi-MV* algebra and strong quasi-Wajsberg* algebra and show that they are term equivalence. Second, we present the representation theorem and the standard completeness of strong quasi-MV* algebras. Moreover, we discuss the properties of terms in the language of Wajsberg* algebras and strong quasi-Wajsberg* algebras. Finally, we establish the logical system associated with strong quasi-Wajsberg* algebra and prove that the logical system is sound and complete.

\end{abstract}


\maketitle
\baselineskip=16.2pt

\dottedcontents{section}[1.16cm]{}{1.8em}{5pt}
\dottedcontents{subsection}[2.00cm]{}{2.7em}{5pt}
\section{Introduction}

Motivated by philosophical and computational problems of vagueness and imprecision, fuzzy logic has become a significant field in non-classical logic. Many researches in this area focus on many-valued logics with linearly ordered truth values and have yielded elegant and deep mathematical theories, thus more and more researchers pay their attentions on fuzzy logic.
It is well-known that the algebraic study of non-classical logic is a general and important tool which provides an abstract insight into the fundamental logical properties.
Therefore, with the extension and generalization of fuzzy logic, many non-classical logical algebras have been introduced such as BL-algebras \cite{7.1}, EQ-algebras \cite{14.1}, MV*-algebras \cite{11}, quantum B-algebras \cite{15.1},  quasi-MV algebras \cite{10} and so on.

Quasi-MV* algebras were introduced in \cite{9} as the generalization of MV*-algebras and quasi-MV algebras. The recent research presents that quasi-MV* algebras are closely related to complex fuzzy logic \cite{8}. Complex fuzzy logic (CFL, for short) is an extension of traditional fuzzy logic (Zadeh's fuzzy logic) that has shown promise in date stream mining (particularly, time series forecasting). Due to the circle structured codomain of the complex membership functions, the study of formal CFL in the narrow sense was developed slowly. In order to promote this study, some authors tried to introduce the ideas and methods of traditional fuzzy logic into CFL and obtained some related results \cite{4,5,7,14,16}. For example, Tamir et al. gave the Cartesian representation of the complex values of complex membership functions in the following way: $u(V,z)=u_{r}(V)+ju_{i}(z)$, where $u_{r}(V), u_{i}(z)$ are in the real unit interval $[0, 1]$ and $\sqrt{-1}=j$ \cite{16}. Based on this representation, Tamir et al. introduced generalized complex propositional fuzzy logic using direct generalization of traditional fuzzy logic. However, as we know,  the Polar representation of any complex number $z$ to its Cartesian representation by
Euler formula is $z = \rho e^{j\theta} =\rho \cos(\theta) + j\rho \sin(\theta)$ where $\rho$ is a real number, $\cos(\theta), \sin(\theta)$ are in the real closed interval $[-1, 1]$, and $\sqrt{-1}=j$. Therefore, following Tamir's ideas,  we consider that the set $[-1, 1]\times[-1, 1]$ is more suitable to express the cartesian representation of CFL.

For this reason, we must mention MV*-algebras. The original investigation of an MV*-algebra was to model the properties of the real interval $[-1,1]$ equipped with the truncated addition $x\oplus y=\max\{-1,\min\{1,x+y\}\}$ and negation $-x$, paralleling work done for MV-algebras \cite{2}. The properties of MV*-algebras and their associated logic $\L^{*}$ were studied further in \cite{12,13}. In \cite{2}, Chang showed that any element in an MV*-algebra is the sum of its positive part and negative part. Moreover, the set of all positive elements of an MV*-algebra can be an MV-algebra which is isomorphic to the algebra of all negative elements of the same MV*-algebra.
Maybe because MV*-algebras are so near to MV-algebras, the progress of MV*-algebras or the logic $\L^{*}$ was rare to see for a long time.
In fact, if we consider the geometry meaning of the closed interval $[-1, 1]$ in the complex plane, it is easy to see that $[-1,1]$ can be regarded as the set of all
points satisfying the distance to the origin is not more than $1$ on the real axis, i.e., $[-1,1]=\{z\in \mathbb{R}||z|\leq 1\}$, thus the set $\{z\in \mathbb{C}||z|\leq1\}$ of all points satisfying the distance to the origin no more than $1$ on the complex plane is a natural generalization of the set $\{z\in \mathbb{R}||z|\leq 1\}=[-1,1]$, and then CFL can be regarded as a generalization of the logic $\L^{*}$. Thereby, we think that the generalization of MV*-algebras will have great role in the study of algebraic structure of CFL.

Another motivation to consider the generalization of MV*-algebras is that, on the basis of the existing research results, the algebraic structures in traditional fuzzy logic can not be easily transplanted to CFL \cite{4,5,14}.
For example, the algebraic product in CFL does not hold the properties in traditional fuzzy logic, the complex fuzzy $S$-implication operator is not continuous while it is continuous in traditional fuzzy logic.
In \cite{3}, Dai constructed quasi-MV algebras in CFL. The standard quasi-MV algebras include the square $\textbf{S}=([0,1]\times [0,1]; \oplus, ',0,1)$  and the disc $\textbf{D}=(\{\langle a,b\rangle\in \mathbb{R}^{2}|(1-2a)^{2}+(1-2b)^{2}\leq 1\};\oplus, ',0,1)$. Dai defined a mapping $F$ from the disc standard quasi-MV algebra $\textbf{D}$ to the universe $\{\langle a,b\rangle\in \mathbb{R}^{2}|a^{2}+b^{2}\leq 1\}=\{z\in \mathbb{C}||z|\leq1\}$  of CFL and showed that the set $\{z\in \mathbb{C}||z|\leq1\}$ can be a quasi-MV algebra.
Since the universe of CFL contains the closed interval $[-1,1]$, we always wish that the algebraic structure of CFL restricted to the set $[-1,1]$ is inherited to the original algebra. Based on Dai's valuable consideration on CFL, interesting enough, we find that quasi-MV* algebras are just ones we need in the study of CFL.

In \cite{8}, we had discussed some properties of quasi-MV* algebras and established the logical system associated with quasi-MV* algebras. In the process of research, we found an interesting subvariety of quasi-MV* algebras which is important for the study of CFL. Hence in this paper, we investigate the properties of this subvariety and its associated logical system further. Weak implicative logic introduced by Cintula and Noguera is a general non-classical logical system including many known logics \cite{2.1}. Although the connectives $\rightarrow$ and $\neg$ in our logical system are similar to $\L^{*}$ and $\L^{*}$ is a weak implicative logic, we notice that the Modus Ponens rule does not hold in the new logical system, so it is not a weak implicative logic in general. The paper is arranged as follows. In Section 2, we recall some definitions and results which are related to quasi-MV* algebras. In Section 3, we introduce the strong quasi-MV* algebra and strong quasi-Wajsberg* algebra and show the term equivalence between them. We also prove the standard completeness of strong quasi-MV* algebras. In Section 4, we establish the logical system associated with strong quasi-Wajsberg* algebras and discuss the soundness and completeness of this logical system.

\section{Preliminary}

In this section, we recall some definitions and results which will be used in what follows.

\begin{defn}\cite{11}
Let $\textbf{B}=\langle B;\oplus,{-},0,1\rangle$ be an algebra of type $\langle2,1,0,0\rangle$. If the following conditions are satisfied for any $x,y,z\in B$,

(MV*1)  $x \oplus y =y \oplus x$,

(MV*2)  $(1\oplus x) \oplus(y \oplus( 1\oplus z))=((1\oplus x) \oplus y) \oplus(1\oplus z)$,

(MV*3)  $x \oplus (-x)=0$,

(MV*4)  $(x \oplus 1)\oplus 1 =1$,

(MV*5)  $x \oplus 0 =x$,

(MV*6)  ${-}(x \oplus y) =(-x) \oplus (-y)$,

(MV*7)  ${-}(-x) =x$,

(MV*8)  $x \oplus y =(x^{+} \oplus y^{+} )\oplus (x^{-}\oplus y^{-})$,

(MV*9)  $(-x \oplus (x \oplus y))^{+} ={-}(x^{+})\oplus (x^{+} \oplus y^{+})$,

(MV*10)  $x \vee y =y \vee x$,

(MV*11)  $x \vee (y\vee z) =(x \vee y)\vee z$,

(MV*12)  $x \oplus (y\vee z) =(x \oplus y)\vee(x\oplus z)$,

\noindent in which ones define
$x^{+}=1 \oplus(-1 \oplus x)$,
$x^{-}=-1 \oplus(1 \oplus x)$, and
$x\vee y=(x^{+}\oplus(-x^{+}\oplus y^{+})^{+})\oplus(x^{-}\oplus(-x^{-}\oplus y^{-})^{+})$, then $\textbf{B}=\langle B;\oplus,{-},0,1\rangle$ is called an \emph{MV*-algebra}. \end{defn}

The variety of MV*-algebras is denoted by $\mathbb{MV^\ast}$.
Given an MV*-algebra \textbf{B}, we can define $x \leq y$ iff $x\vee y=y$ for any $x,y \in B$. If $0\leq x$, then the element $x$ is called \emph{positive}. If $x$ is a positive element, we also say $x\geq 0$. If $x\leq 0$, then the element $x$ is called \emph{negative}. We denote $B_{\geq 0}=\{ x\in B | x\geq 0\}$ and $B_{\leq 0}=\{ x\in B | x\leq 0 \}$. Moreover, we denote $B^+ = \{ x^+ \in B| x\in B \}$ and $B^- = \{ x^- \in B| x\in B \}$. In \cite{11}, authors showed that $B_{\geq 0} = B^+$ and $B_{\leq 0} = B^-$.

\begin{defn}\cite{8}\label{wdy}
Let $\textbf{M} = \langle M; \to, \neg, 1 \rangle$ be an algebra of type $\langle 2,1,0 \rangle$. If the following conditions are satisfied for any $x, y, z \in M$,

(W*1)  $x\to y = (\neg y) \to (\neg x)$,

(W*2)  $(x\to 1)\to ((y\to 1)\to z)=(y\to 1)\to ((x\to 1)\to z)$,

(W*3)  $(1\to x)\to 1=1$,

(W*4)  $(y\to y)\to x=x$,

(W*5)  $x\to y = (y^+ \to x^-)\to (x^+ \to y^-)$,

(W*6)  $\neg(x\to y) = y\to x$,

(W*7)  $\neg \neg x=x$,

(W*8)  $(x\to (\neg x \to y))^+ = x^+ \to (\neg x^+ \to y^+)$,

(W*9)  $x\vee y = y\vee x$,

(W*10)  $x\vee(y\vee z)= (x\vee y)\vee z$,

(W*11)  $x\to (y\vee z)= (x\to y)\vee (x\to z)$,

\noindent in which ones define $x^+ = (x\to 1)\to 1$, $x^- = (x\to \neg 1)\to \neg 1$, and $x\vee y = ((x^+ \to y^+)^+ \to (\neg x)^-) \to ((y^- \to x^-)^- \to x^-)$, then $\textbf{M} = \langle M; \to, \neg, 1 \rangle$ is called a \emph{Wajsberg* algebra}.
\end{defn}

The variety of Wajsberg* algebras is denoted by $\mathbb{W^\ast}$. In \cite{8}, authors had proved that Wajsberg* algebras and MV*-algebras are term equivalence. Hence many results about MV*-algebras can be directly built to Wajsberg* algebras.

The logic $\L^*$ associated with MV*-algebras, was introduced by Chang in \cite{2} and investigated further in \cite{12,13}. Below we list the axioms and rules of $\L^*$.

Let $F_{\scriptscriptstyle L^*}(V)$ be the formulas set generated by all propositional variables set $V$ with logical connectives $\to$, $\neg$ and a constant $1$. Then $\langle F_{\scriptscriptstyle L^*}(V); \to, \neg, 1 \rangle$ is a free algebra.
For any $p, q\in F_{\scriptscriptstyle L^*}(V)$, the axioms of $\L^{*}$ are defined as follows:

\noindent\textbf{Axioms schemas}

(P1) $(p\to q)\leftrightarrow (\neg q\to \neg p)$,

(P2) $p \leftrightarrow ((q\to q)\to p)$,

(P3) $\neg(p\to q)\leftrightarrow (q\to p)$,

(P4) $p\to 1$,

(P5) $1\leftrightarrow ((1\to p)\to 1)$,

(P6) $((p\to 1)\to((q\to 1)\to r))\to ((q\to 1)\to((p\to 1)\to r))$,

(P7)  $(p\to q)\leftrightarrow((q^{+}\to p^{-})\to (p^{+}\to q^{-}))$ where $p^+ = (p \to 1)\to 1 $ and $p^- = (p \to \neg 1)\to \neg 1 $,

(P8) $(p\to (\neg p\to q))^+\leftrightarrow (p^+\to (\neg p^+\to q^+))$,

(P9) $(p\to (q\vee r))\leftrightarrow((p\to r)\vee (p\to q))$ where $p\vee q = (((p^{+}\to q^{+})^{+})\to(\neg p)^{-} )\to((q^{-}\to p^{-})^{-}\to p^{-})$,

(P10) $(p\vee (q\vee r))\leftrightarrow ((p\vee q)\vee r)$.

The deduction rules of $\L^{*}$ are as follows.

\noindent\textbf{Rules of deduction}

(R1) $ p, p\to q\vdash q$,

(R2) $p\to q, r\to t\vdash (q\to r)\to (p\to t)$,

(R3) $p\vdash p^-$.

Recall that a valuation of $\L^*$ is a mapping $e:F_{\scriptscriptstyle L^*}(V) \rightarrow [-1,1]$ satisfying for any $p,q\in F_{\scriptscriptstyle L^*}(V)$, $e(p\oplus q)=\max \{ -1,\min \{ e(q)+e(p), 1 \} \}$ and $e(\neg p)=-e(p)$.
According to \cite{2} and \cite{13}, we have that any valuation of $\L^*$ can be seen as a homomorphism from $\textbf{F}_{\scriptscriptstyle L^*}(\textbf{V})$ to MV*-algebra. Moreover, for any $q \in F_{\scriptscriptstyle L^*}(V)$, we have that $\vdash_{\scriptscriptstyle L^*} q$ iff for any valuation $e$, $e(q) \geq 0$.
Generally, $q_1,\cdots, q_n \vdash_{\scriptscriptstyle L^*} q$ iff for any valuation $e$, $e(q_1)\geq 0 \& \cdots \& e(q_n) \geq 0 \Rightarrow e(q)\geq 0 $.
Hence, from the term equivalence between MV*-algebras and Wajsberg* algebras, we can get that $q_1,\cdots, q_n \vdash_{\scriptscriptstyle L^*} q$ iff for any valuation $e$, $e(q_1)= (c_1 \to 1)\to 1 \& \cdots \& e(q_n)=(c_n \to 1)\to 1 \Rightarrow e(q)=(c\to 1)\to 1 $, where $c_1, \cdots, c_n,c$ are in Wajsberg* algebra.

\begin{defn}\cite{9}
Let $\textbf{A} = \langle A; \oplus, {-}, ^{+}, ^{-}, 0, 1 \rangle$ be an algebra of type $\langle 2, 1, 1, 1, 0, 0 \rangle$. If the following conditions are satisfied for any $x, y, z \in A$,

(QMV*1) $x\oplus y=y\oplus x$,

(QMV*2) $(1\oplus x)\oplus(y\oplus(1\oplus z))=((1\oplus x)\oplus y)\oplus(1\oplus z)$,

(QMV*3) $(x\oplus 1)\oplus 1=1$,

(QMV*4) $(x\oplus y)\oplus 0=x\oplus y$,

(QMV*5) $x^{+}\oplus 0 = (x\oplus 0)^{+}=1 \oplus(-1 \oplus x)$,

\hspace{1.6cm} $x^{-}\oplus 0=(x\oplus 0)^{-} = -1 \oplus(1 \oplus x)$,

(QMV*6) $x\oplus y=(x^{+}\oplus y^{+})\oplus (x^{-}\oplus y^{-})$,

(QMV*7) $0=-0$,

(QMV*8) $x\oplus (-x)=0$,

(QMV*9) $-(x\oplus y)=(-x)\oplus (-y)$,

(QMV*10) $-(-x)=x$,

(QMV*11) $(-x\oplus(x\oplus y))^{+}=-x^{+}\oplus (x^{+}\oplus y^{+})$,

(QMV*12) $x\vee y=y\vee x$,

(QMV*13) $x\vee (y\vee z)=(x\vee y)\vee z$,

(QMV*14) $x\oplus (y\vee z)=(x\oplus y)\vee(x\oplus z)$,

in which ones define $x\vee y = (x^{+}\oplus(-x^{+} \oplus y^{+})^{+})\oplus(x^{-}\oplus (-x^{-}\oplus y^{-})^{+})$, then $\textbf{A}=\langle A;\oplus,{-},^{+},^{-},0,1\rangle$ is called a \emph{quasi-MV* algebra}.
\end{defn}

The variety of quasi-MV* algebras is denote by $\mathbb{QMV^\ast}$. In any quasi-MV* algebra, we consider that the operations $^+$ and $^-$ (which have the same priority) have priority to operations $\oplus$ and $-$, the operation $-$ has priority to the operation $\oplus$.

Let $\textbf{A} = \langle A; \oplus, {-}, ^{+}, ^{-}, 0, 1 \rangle$ be a quasi-MV* algebra. For any $x,y \in A$, we define a relation $x \leq y$ iff $x\vee y = y\oplus 0$. Then the relation $\leq$ is quasi-ordering. Moreover, we denote $R(A)=\{ x\in A\mid x\oplus 0=x \}$, then $\langle R(A);\oplus,-,0,1 \rangle$ is an MV*-algebra \cite{9}.

\begin{defn}\cite{8}
Let $\textbf{W}=\langle W;\to,\neg,^{+},^{-},1\rangle$ be an algebra of type $\langle2,1,1,1,0\rangle$. If the following conditions are satisfied for any ${x,y,z}\in W$,

(QW*1) $x\to y=(\neg y)\to (\neg x)$,

(QW*2) $(x\to 1)\to((y\to 1)\to z)=(y\to 1)\to((x\to 1)\to z)$,

(QW*3) $(1\to x)\to 1=1$,

(QW*4) $(z\to z)\to (x\to y)=x\to y$,

(QW*5) $(1\to 1)\to x^{+}=((1\to 1)\to x)^{+}=(x\to 1)\to 1$ and

\hspace{1.4cm} $(1\to 1)\to x^{-}=((1\to 1)\to x)^{-}=(x\to \neg 1)\to \neg 1$,

(QW*6) $x\to y=(y^{+}\to x^{-})\to (x^{+}\to y^{-})$,

(QW*7) $\neg(x\to y)= y\to x$,

(QW*8) $\neg\neg x=x$,

(QW*9) $(x\to (\neg x\to y))^{+}=x^{+}\to (\neg x^{+}\to y^{+})$,

(QW*10) $x\vee y=y\vee x$,

(QW*11) $x\vee (y\vee z)=(x\vee y)\vee z$,

(QW*12) $x\to (y\vee z)=(x\to y)\vee (x\to z)$,

\noindent in which ones define $x\vee y = ((x^{+}\to y^{+})^{+}\to(\neg x)^{-}) \to((y^{-}\to x^{-})^{-}\to x^{-})$, then $\textbf{W}=\langle W;\to,\neg,^{+},^{-},1\rangle$ is called a \emph{quasi-Wajsberg* algebra}.
\end{defn}

The variety of quasi-Wajsberg* algebras is denoted by $\mathbb{QW^\ast}$.
In any quasi-Wajsberg* algebra, we consider that the operations $^+$ and $^-$ (which have the same priority) have priority to operations $\to$ and $\neg$, the operation $\neg$ has priority to the operation $\to$.

\begin{prop}\label{lox1}
Let $\textbf{W}=\langle W;\to,\neg,^{+},^{-},1\rangle$ be a quasi-Wajsberg* algebra. Then for any $x,y,z\in W$, we have

\emph{(1)} $\neg x\to y=\neg y\to x$ and $x\to \neg y=y\to\neg x$,

\emph{(2)} $x\to x = y\to y$,

\emph{(3)} $(z\to z)\to \neg(x\to y)= \neg(x\to y)$.
\end{prop}

\begin{proof}
(1) and (2) follow from Proposition 3.1 in \cite{8}, we only need to prove (3). For any $x,y,z\in W$, we have $(z\to z)\to \neg(x\to y)= (z\to z)\to (y\to x)=y\to x=\neg(x\to y)$ by (QW*7) and (QW*4).
\end{proof}

Let $\textbf{W}=\langle W;\to,\neg,^{+},^{-},1\rangle$ be a quasi-Wajsberg* algebra and denote $\textbf{0}=1\to 1$. Then by Proposition \ref{lox1}(2), we have that $\textbf{0} = x\to x$ for any $x \in W$. Moreover, we denote $R(W)=\{ x\in W\mid \textbf{0} \to x = x \}$, then $\langle R(W); \to, \neg,1 \rangle$ is a Wajsberg* algebra.

\begin{defn}\cite{8}
Let $\textbf{W}=\langle W;\to,\neg,^{+},^{-},1\rangle$ be a quasi-Wajsberg* algebra. Then $\textbf{W}$ is called \emph{flat}, if it satisfies the equation $\mathbf{0}= 1$.
\end{defn}

\begin{prop}\emph{\cite{8}}\label{corr1}
Let $\mathbf{A}=\langle A;\oplus,{-},^{+},^{-},0,1\rangle$ be a quasi-MV* algebra. If for any $x,y\in A$, we define $x\to y=-x\oplus y$ and $\neg x=-x$, then $f(\mathbf{A})=\langle A;\to,\neg,^{+},^{-},1\rangle$ is a quasi-Wajsberg* algebra, where $x\vee y = ((x^{+}\to y^{+})^{+}\to(\neg x)^{-}) \to((y^{-}\to x^{-})^{-}\to x^{-})$.
\end{prop}

\begin{prop}\emph{\cite{8}}\label{corr2}
Let $\mathbf{W}=\langle W;\to,\neg,^{+},^{-},1\rangle$ be a quasi-Wajsberg* algebra. If for any $x,y\in W$, we define $0=x\to x$, $x\oplus y=\neg x\to y$ and $-x=\neg x$, then $g(\mathbf{W})=\langle W;\oplus,{-},^{+},^{-},0,1\rangle$ is a quasi-MV* algebra, where $x\vee y = (x^{+}\oplus(-x^{+} \oplus y^{+})^{+})\oplus(x^{-}\oplus (-x^{-}\oplus y^{-})^{+})$.
\end{prop}

\begin{thm}\emph{\cite{8}}\label{T1}
For any quasi-MV* algebra $\textbf{A}= \langle A; \oplus,-,^+,^-,0,1 \rangle$ and quasi-Wajsberg* algebra $\textbf{W}= \langle W; \to, \neg, ^{+},^{-},1 \rangle$, we have

\emph{(1)} $gf(\textbf{A})=\textbf{A}$,

\emph{(2)} $fg(\textbf{W})=\textbf{W}$.

So $f$ and $g$ are mutually inverse correspondence.
\end{thm}

For convenience, we abbreviate a quasi-MV* algebra $\textbf{A}= \langle A; \oplus,-,^+,^-,0,1 \rangle$ as $\textbf{A}$ and a quasi-Wajsberg* algebra $\textbf{W}= \langle W; \to, \neg, ^{+},^{-},1 \rangle$ as $\textbf{W}$. Based on the previous equivalence theorem, we can restart the known algebraic results for one algebra in terms of the other without providing any additional proof. Hence in the following we list some properties of quasi-MV* algebras while the similar results of quasi-Wajsberg* algebras had been proved in \cite{8}.

Let \textbf{A} be a quasi-MV* algebra and $\theta$ be a congruence on \textbf{A}. For any $x \in A$, the equivalence class of $x$ with respect to $\theta$ is denoted by $x/\theta = \{ y\in A \mid \langle x,y \rangle \in \theta \}$ and the set of all equivalence classes of elements in $A$ is denoted by $A/\theta$. For any $x/\theta, y/\theta \in A/\theta$, we define the operations on $ A/\theta$ as follows

$(x/\theta) \oplus_{\scriptscriptstyle A/\theta} (y/\theta) = (x \oplus y)/\theta$,

 $(x/\theta)^{+_{_{A/\theta}}}= x^+ /\theta$,

 $(x/\theta)^{-_{_{A/\theta}}}= x^- /\theta$,

 $-_{\scriptscriptstyle A/\theta} (x/\theta) = (-x)/\theta$.

\begin{prop}
Let \textbf{A} be a quasi-MV* algebra and $\theta$ be a congruence on \textbf{A}. Then $\textbf{A}/\theta = \langle A/\theta; \oplus_{\scriptscriptstyle A/\theta}, -_{\scriptscriptstyle A/\theta},  ^{+_{_{A/\theta}}},  ^{-_{_{A/\theta}}} , 0/\theta, 1/\theta\rangle$ is a quasi-MV* algebra.
\end{prop}

\begin{prop} \label{01}
Let \textbf{A} be a quasi-MV* algebra. For any $x, y \in A$, we define
\begin{center}
$\langle x, y \rangle \in \mu$ iff  $x\leq y$ and $y\leq x$,
\end{center}
\begin{center}
$\langle x,y \rangle \in \tau$ iff  $x=y$ or $x,y \in R(A)$.
\end{center}
Then we have

\emph{(1)} $\mu$ and $\tau$ are congruences on \textbf{A},

\emph{(2)} $\mu \cap \tau = \Delta$ (i.e., diagonal relation),

\emph{(3)} $\textbf{A}/\mu = \langle A/\mu; \oplus_{\scriptscriptstyle A/\mu}, -_{\scriptscriptstyle A/\mu}, ^{+_{_{A/\mu}}}, ^{-_{_{A/\mu}}}, 0/\mu, 1/\mu \rangle$ is an MV*-algebra,

\emph{(4)} $\textbf{A}/\tau = \langle A/\tau; \oplus_{\scriptscriptstyle A/\tau}, -_{\scriptscriptstyle A/\tau}, ^{+_{_{A/\tau}}}, ^{-_{_{A/\tau}}}, 0/\tau, 1/\tau \rangle$ is a flat quasi-MV* algebra.
\end{prop}

\section{The standard completeness and term equivalence of strong quasi-MV* algebras}

In this section, we introduce the strong quasi-MV* algebras and investigate their related properties. We prove that every strong quasi-MV* algebra can be embedded into the direct product of an MV*-algebra and a flat strong quasi-MV* algebra. Moreover, we show the standard completeness theorems for flat strong quasi-MV* algebras and strong quasi-MV* algebras, respectively.
We also study strong quasi-Wajsberg* algebras as the term equivalence of strong quasi-MV* algebras and discuss the terms in the languages of  Wajsberg* algebras and strong quasi-Wajsberg* algebras.

\begin{defn}\label{sqmv}
A quasi-MV* algebra $\textbf{Q}= \langle Q;\oplus,-,^+,^-,0,1 \rangle$ is called a \emph{strong quasi-MV* algebra}, if it satisfies the equations $x^+ = x^+ \oplus 0$ and $x^- = x^- \oplus 0$ for any $x \in Q$.
\end{defn}

\begin{rem}\label{remark1}
Let $\textbf{Q}= \langle Q;\oplus,-,^+,^-,0,1 \rangle$ be a strong quasi-MV* algebra. Then for any $x \in Q$, we can get that $x^+ = 1\oplus (-1 \oplus x)$ and $x^- = -1 \oplus (1 \oplus x)$ from (QMV*5). So it is seen that similar to MV*-algebra, the operations $^+$ and $^-$ in a strong quasi-MV* algebra can be determined by operations $\oplus$ and $-$.
\end{rem}

\begin{exam}
Let $Q= [-1,1]\times [0,1] \subseteq \mathbb{R}^{2}$. For any $\langle a,b \rangle, \langle c,d \rangle \in Q$, we define the operations on $Q$ as follows:

$\langle a,b \rangle \oplus \langle c,d \rangle = \langle \max\{ -1, \min\{1, a+c \} \} , \frac{1}{2} \rangle$,

$- \langle a,b \rangle = \langle -a, 1-b \rangle$,

$\langle a,b \rangle ^+ = \langle \max\{0,a \}, \frac{1}{2} \rangle$,

$\langle a,b \rangle ^- = \langle \min\{0,a \}, \frac{1}{2} \rangle$,

$\textbf{0}= \langle 0, \frac{1}{2} \rangle$, $\textbf{1}= \langle 1, \frac{1}{2} \rangle$.

Then we can check that $\langle Q; \oplus, -, ^+, ^-, \textbf{0}, \textbf{1} \rangle$ is a quasi-MV* algebra. For any $\langle a,b \rangle \in Q$, we have that $\langle a,b \rangle ^+ \oplus \textbf{0} = \langle \max\{ 0,a \},  \frac{1}{2} \rangle \oplus \langle 0, \frac{1}{2} \rangle = \langle \max\{ 0,a \}, \frac{1}{2} \rangle = \langle a,b \rangle ^+$ and $\langle a,b \rangle ^- \oplus \textbf{0} = \langle \min\{ 0,a \}, \frac{1}{2} \rangle \oplus \langle 0, \frac{1}{2} \rangle = \langle \min\{ 0,a \}, \frac{1}{2} \rangle = \langle a,b \rangle ^-$. Hence $\langle Q; \oplus, -, ^+, ^-, \textbf{0}, \textbf{1} \rangle$ is a strong quasi-MV* algebra. Note that $\langle a,b \rangle \oplus \textbf{0} = \langle a,b \rangle \oplus \langle 0, \frac{1}{2} \rangle = \langle a, \frac{1}{2} \rangle \neq \langle a,b \rangle$ whenever $b \neq \frac{1}{2}$, so $\langle Q; \oplus, -, ^+, ^-, \textbf{0}, \textbf{1} \rangle$ is not an MV*-algebra.
\end{exam}

\begin{exam} Let $S^*= [-1,1]\times [-1,1] \subseteq \mathbb{R}^{2}$. For any $\langle a,b \rangle, \langle c,d \rangle \in S^*$, we define the operations on $S$ as follows:

$\langle a,b \rangle \oplus \langle c,d \rangle = \langle \max\{-1, \min\{ 1, a+c\}\}, 0 \rangle$,

$-\langle a,b \rangle = \langle -a, -b \rangle$,

$\langle a,b \rangle ^+ = \langle \max\{ 0,a \}, 0 \rangle$,

$\langle a,b \rangle ^- = \langle \min\{ 0,a \}, 0 \rangle$,

$\textbf{0}= \langle 0,0 \rangle, \textbf{1}= \langle 1,0 \rangle$.

Then $\textbf{S*} = \langle S^*; \oplus, -, ^+, ^-, \textbf{0}, \textbf{1} \rangle$ is a strong quasi-MV* algebra. We call $\textbf{S*}$ the \emph{square standard strong quasi-MV* algebra}. Especially, we denote $D^* = \{ \langle a,b \rangle \in \mathbb{R}^{2} \mid a^2 + b^2 \leq 1 \}$. Then $\textbf{D*}=\langle D^*; \oplus_{\scriptscriptstyle D^*}, -_{\scriptscriptstyle D^*},$ $^{+_{D^*}}, ^{-_{D^*}}, \textbf{0}, \textbf{1} \rangle$ is the subalgebra of \textbf{S*}, where the operations $\oplus_{\scriptscriptstyle D^*}$, $-_{\scriptscriptstyle D^*}$, $^{+_{D^*}}$ and $^{-_{D^*}}$ are those of $\textbf{S*}$ restricted to $D^*$. We call $\textbf{D*}$ the \emph{disk standard strong quasi-MV* algebra}.
\end{exam}

In the following, we abbreviate a strong quasi-MV* algebra $\textbf{Q} = \langle Q; \oplus, - ^+, ^-, 0,1 \rangle$ as $\textbf{Q}$. The variety of strong quasi-MV* algebras is a subvariety of quasi-MV* algebras and denoted by $\mathbb{SQMV^\ast}$.

\begin{defn}\label{F1}
Let $\textbf{Q}$ be a strong quasi-MV* algebra. Then \textbf{Q} is called \emph{flat}, if it satisfies the equation $0= 1$.
\end{defn}

Below we denote $\textbf{F}$ for the flat strong quasi-MV* algebra and the variety of flat strong quasi-MV* algebras by $\mathbb{FSQMV^\ast}$.

\begin{exam}
Let $\textbf{A} = \langle A; \oplus _{\scriptscriptstyle A}, -_{\scriptscriptstyle A}, ^{+_{_{A}}}, ^{-_{_{A}}}, 0_{\scriptscriptstyle A}, 1_{\scriptscriptstyle A} \rangle$ be a quasi-MV* algebra and let $k\notin A$ if $-_{\scriptscriptstyle A}$ has no fixpoint over $R(A)$, otherwise let $k$ be such a fixpoint. Then the $k$-flattening of \textbf{A} is the structure
$$ \textbf{F}(\textbf{A},k) = \langle A\cup \{ k \}; \oplus _{\scriptscriptstyle F}, -_{\scriptscriptstyle F}, ^{+_{_{F}}}, ^{-_{_{F}}}, \textbf{0}_{\scriptscriptstyle F}, \textbf{1}_{\scriptscriptstyle F} \rangle, $$

where

$ \textbf{0}_{\scriptscriptstyle F} = \textbf{1}_{\scriptscriptstyle F} = k $,

for any $a, b \in A\cup \{k \}$, $a \oplus _{\scriptscriptstyle F} b = k$, $a^{+_{_{F}}} = k$, $a^{-_{_{F}}} = k$,

for any $a \in A - \{k \}$, $-_{_{F}} a = -_{_{A}} a$, $-_{_{F}} k =k$.

\noindent It is easy to see that $\textbf{F}(\textbf{A},k)$ is a flat strong quasi-MV$^\ast$ algebra.
\end{exam}

\begin{exam}
The \emph{standard flat strong quasi-MV* algebra} is the 0-flattening of the standard MV* algebra $\textbf{MV*}_{[-1,1]}$, i.e., the algebra
$$\textbf{F}(\textbf{MV*}_{[-1,1]}, 0) = \langle [-1, 1]; \oplus_{\scriptscriptstyle F}, -_{\scriptscriptstyle F}, ^{+_{_{F}}}, ^{-_{_{F}}}, \textbf{0}, \textbf{1} \rangle,$$

\noindent where for any $a, b\in [-1,1]$, $a \oplus_{\scriptscriptstyle F} b = 0$, $-_{\scriptscriptstyle F} a = -a$, $a^{+_{_{F}}} = 0$, $a^{-_{_{F}}} = 0$, and $\textbf{0}=\textbf{1}=0$.
\end{exam}

\begin{lem}\label{f0}
Let \textbf{F} be a flat strong quasi-MV* algebra. Then for any $x, y \in F$, $x \oplus y = 0$.
\end{lem}

In order to show the standard completeness result for $\mathbb{SQMV^\ast}$, we first complete the proof of the direct embedding theorem, i.e., every strong quasi-MV* algebra is embeddable into the direct product of an MV*-algebra and a flat strong quasi-MV* algebra.

Given $\textbf{W}=\langle W; \oplus_{\scriptscriptstyle W}, -_{\scriptscriptstyle W}, ^{+_{_ {W}}}, ^{-_{_ {W}}}, 0_{\scriptscriptstyle W}, 1_{\scriptscriptstyle W} \rangle$ and $\textbf{V}=\langle V; \oplus_{\scriptscriptstyle V}, -_{\scriptscriptstyle V}, ^{+_{_ {V}}}, ^{-_{_ {V}}},0_{\scriptscriptstyle V}, 1_{\scriptscriptstyle V} \rangle$ are strong quasi-MV* algebras. Define the operations $\oplus$, $-$, $^+$ and $^-$ as coordinate-wise on $W\times V$. Then $\textbf{W}\times \textbf{V} = \langle W\times V; \oplus,-,^+,^-, \langle 0_{\scriptscriptstyle W},0_{\scriptscriptstyle V} \rangle , \langle 1_{\scriptscriptstyle W}, 1_{\scriptscriptstyle V}\rangle \rangle$ is a strong quasi-MV* algebra. Moreover, a mapping $f: W \rightarrow V$ is called a \emph{strong quasi-MV* homomorphism} from \textbf{W} to \textbf{V}, if the following conditions are satisfied for any $x,y \in W$,

 (1) $f(1_{\scriptscriptstyle W}) = 1_{\scriptscriptstyle V}$ and $f(0_{\scriptscriptstyle W}) = 0_{\scriptscriptstyle V}$,

 (2) $f(x \oplus_{\scriptscriptstyle W} y)= f(x)\oplus_{\scriptscriptstyle V} f(y)$,

 (3) $f(-_{\scriptscriptstyle W} x) = -_{\scriptscriptstyle V} f(x)$.

It is easy to see that $f(x^{+_{\scriptscriptstyle W}})=(f(x))^{+_{\scriptscriptstyle V}}$ and $f(x^{-_{\scriptscriptstyle W}})=(f(x))^{-_{\scriptscriptstyle V}}$, if $f$ is a strong quasi-MV* homomorphism from $\textbf{W}$ to $\textbf{V}$.

\begin{lem}
Let \textbf{W} and \textbf{V} be strong quasi-MV* algebras. If a mapping $f: W\rightarrow V$ is the homomorphism from \textbf{W} to \textbf{V}, then $f(W)=\{ f(x) | x\in W \}$ is a subalgebra of \textbf{V}.
\end{lem}

\begin{prop}\label{02}
Let \textbf{Q} be a strong quasi-MV* algebra. Then there exist an MV*-algebra \textbf{B} and a flat strong quasi-MV* algebra \textbf{F} such that \textbf{Q} can be embedded into the direct product $\textbf{B} \times \textbf{F}$. Such an embedding is an isomorphism if $\textbf{Q}\in \mathbb{MV^*}$ or $\textbf{Q}\in \mathbb{FSQMV^*}$.
\end{prop}

\begin{proof}
Denote $\textbf{B}= \textbf{Q}/\mu$ and $\textbf{F}= \textbf{Q}/\tau$, where $\mu$ and $\tau$ are defined in Proposition \ref{01}. Then we have that $\textbf{B}$ is an MV*-algebra and $\textbf{F}$ is a flat quasi-MV* algebra. Since \textbf{Q} is a strong quasi-MV* algebra, we have $(x/\tau)^{+_{Q/\tau}} \oplus_{\scriptscriptstyle Q/\tau} (0/\tau) = (x^+ /\tau ) \oplus_{\scriptscriptstyle Q/\tau} (0/\tau) = (x^+ \oplus 0)/\tau = x^+ /\tau = (x/\tau)^{+_{_{Q/\tau}}}$ and $(x/\tau)^{-_{_{Q/\tau}}} \oplus_{\scriptscriptstyle Q/\tau} (0/\tau) = (x^- /\tau) \oplus_{\scriptscriptstyle Q/\tau} (0/\tau) = (x^- \oplus 0)/\tau = x^- /\tau = (x/\tau)^{-_{_{Q/\tau}}}$, so $\textbf{F}$ is a flat strong quasi-MV* algebra.
Define the mapping $f : Q \rightarrow B\times F$ by $f(x)=\langle x/\mu, x/\tau \rangle$ for any $x \in Q$. For any $x,y\in Q$, if $f(x)=f(y)$, then $\langle x/\mu, x/\tau \rangle = \langle y/\mu, y/\tau \rangle$, it turns out that $x/\mu = y/\mu$ and $x/\tau = y/\tau$, so $\langle x,y \rangle \in \mu$ and $\langle x,y \rangle \in \tau$. Because $\mu \cap \tau = \Delta$ by Proposition \ref{01}, we have that $x=y$ and then $f$ is injective. Moreover, we have $f(x \oplus_{\scriptscriptstyle Q} y) = \langle (x\oplus_{\scriptscriptstyle Q} y)/\mu, (x\oplus_{\scriptscriptstyle Q} y)/\tau \rangle = \langle x/\mu \oplus_{\scriptscriptstyle Q/\mu} y/\mu, x/\tau \oplus_{\scriptscriptstyle Q/\tau} y/\tau \rangle = \langle x/\mu, x/\tau \rangle \oplus_{\scriptscriptstyle Q/\mu \times Q/\tau} \langle y/\mu, y/\tau \rangle = f(x) \oplus_{\scriptscriptstyle Q/\mu \times Q/\tau} f(y)$ and $f(-_{\scriptscriptstyle Q}x) = \langle (-_{\scriptscriptstyle Q} x)/\mu, (-_{\scriptscriptstyle Q} x)/\tau \rangle = \langle -_{\scriptscriptstyle Q/\mu} (x/\mu), -_{\scriptscriptstyle Q/\tau} (x/\tau) \rangle = -_{\scriptscriptstyle Q/\mu \times Q/\tau} \langle x/\mu, x/\tau \rangle = -_{\scriptscriptstyle Q/\mu \times Q/\tau} f(x)$. Besides, $f(0_{\scriptscriptstyle Q}) = \langle 0_{\scriptscriptstyle Q/\mu}, 0_{\scriptscriptstyle Q/\tau} \rangle = 0_{\scriptscriptstyle Q/\mu \times Q/\tau}$ and $f(1_{\scriptscriptstyle Q})= \langle 1_{\scriptscriptstyle Q/\mu}, 1_{\scriptscriptstyle Q/\tau} \rangle = 1_{\scriptscriptstyle Q/\mu \times Q/\tau}$. Hence $f$ is a strong quasi-MV* homomorphism
and then \textbf{Q} is embedded into the direct product $\textbf{B} \times \textbf{F}$.

Moreover, if the mapping $f$ is surjective, then for any $\langle x/\mu, y/\tau \rangle \in Q/\mu \times Q/\tau$, there is an element $z\in Q$ such that $f(z) = \langle z/\mu, z/\tau \rangle=\langle x/\mu, y/\tau \rangle$, it follows that $z/\mu=x/\mu$ and $z/\tau=y/\tau$, so $z\in x/\mu \cap y/\tau$ and then $x/\mu \cap y/\tau\neq \emptyset$. Conversely, if for any $x,y\in Q$, $x/\mu \cap y/\tau \neq \emptyset$, then it is easy to see that $f$ is surjective following the above proof.
So, if $Q\in \mathbb{MV^*}$, then $\tau = \nabla$ (all relation) and then $x/\mu \cap y/\tau \neq \emptyset$, so $f$ is surjective. If $Q\in \mathbb{FSQMV^*}$, then $\mu = \nabla$ and then $x/\mu \cap y/\tau \neq \emptyset$, we also have that $f$ is surjective. Hence $f$ is an isomorphism if $Q\in \mathbb{MV^*}$ or $Q\in \mathbb{FSQMV^*}$.
\end{proof}

\begin{exam}
The algebra $\textbf{S*} = \langle S^{*}; \oplus, -, ^+, ^-, \textbf{0}, \textbf{1} \rangle$ in Example 3.2 is a strong quasi-MV* algebra.
We can get that $\textbf{S*}$ is just the direct product of $\textbf{MV*}_{[-1,1]}$ and $\textbf{F}(\textbf{MV*}_{[-1,1]}, 0)$.
\end{exam}

Below we prove the standard completeness theorem for $\mathbb{FSQMV^*}$ with respect to the standard flat strong quasi-MV* algebra $\textbf{F}(\textbf{MV*}_{[-1,1]}, 0)$.

\begin{prop}\label{04}
Let $t$ and $s$ be terms in the language of strong quasi-MV* algebra. Then we have $\mathbb{FSQMV^*} \models t \thickapprox s$ iff $\textbf{F}(\textbf{MV*}_{[-1,1]},0) \models t \thickapprox s$.
\end{prop}

\begin{proof}
If $\mathbb{FSQMV^\ast} \models t \thickapprox s$,  then we have $\textbf{F}(\textbf{MV*}_{[-1,1]},0) \models t \thickapprox s$.

Now, if $\textbf{F}(\textbf{MV*}_{[-1,1]},0) \models t \thickapprox s$ and we suppose $\mathbb{FSQMV^\ast} \nvDash t \thickapprox s$, then there is a flat strong quasi-MV* algebra \textbf{F} such that $\textbf{F} \nvDash t \thickapprox s$. We distinguish several cases in order to discuss the question.

(1) Both $t$ and $s$ contain at least an occurrence of $\oplus$. We can assign $x_i $ any value $a_i \in F$ $(i=1,2,\cdots ,n)$, then $t^{\scriptscriptstyle \textbf{F}}(a_1, \cdots, a_n)= 0^{\scriptscriptstyle \textbf{F}}=s^{\scriptscriptstyle \textbf{F}}(a_1, \cdots, a_n)$ from Lemma \ref{f0}, against the hypothesis.

(2) Either $t$ or $s$ contains at least an occurrence of $\oplus$.  Without loss of generality, we assume that $t$ contains at least an occurrence of $\oplus$ and $s$ does not contain an occurrence of $\oplus$. Then $s$ is one of the following forms: $s$ is a variable $x_i$ and contains $k$ $(0\leq k)$ occurrences of $-$, or $s$ is a constant and contains $k$ $(0\leq k)$ occurrences of $-$. If $s$ is the former, we can assign $x_i$ the value $0\neq b_i \in [-1,1]$ and the variables $ x_1,\cdots, x_{i-1}, x_{i+1},\cdots, x_n $ any values $b_1, \cdots,b_{i-1}, b_{i+1},\cdots, b_n\in[-1,1]$, then  $t^{\scriptscriptstyle \textbf{F}(\textbf{MV*}_{[-1,1]},0)}(b_1, \cdots, b_n)= 0 \neq \pm b_{i}= s^{\scriptscriptstyle \textbf{F}(\textbf{MV*}_{[-1,1]},0)}(b_i)$.
This is a contradiction with $\textbf{F}(\textbf{MV*}_{[-1,1]},0) \models t \thickapprox s$.
If $s$ is the latter, for any value $a_i\in F$, we have $s(a_1, \cdots, a_n)=0$, $1$, or $-1$. If $s(a_1, \cdots, a_n)=0$, then
$s^{\scriptscriptstyle \textbf{F}}=0^{\scriptscriptstyle \textbf{F}}$ and then we have $t^{\scriptscriptstyle \textbf{F}} = s^{\scriptscriptstyle \textbf{F}}$, against the hypothesis.
If $s(a_1, \cdots , a_n) = 1$, then $s^{\scriptscriptstyle \textbf{F}}=1^{\scriptscriptstyle \textbf{F}} = 0^{\scriptscriptstyle \textbf{F}}$ and then we have  $t^{\scriptscriptstyle \textbf{F}} = s^{\scriptscriptstyle \textbf{F}}$, against the hypothesis. If $s(a_1, \cdots, a_n) = -1$, then $s^{\scriptscriptstyle \textbf{F}}=-1^{\scriptscriptstyle \textbf{F}} = -0^{\scriptscriptstyle \textbf{F}}=0^{\scriptscriptstyle \textbf{F}}$ and then we have $t^{\scriptscriptstyle \textbf{F}} = s^{\scriptscriptstyle \textbf{F}}$, against the hypothesis.

(3) Neither $t$ nor $s$ contains any occurrence of $\oplus$, then $t$ and $s$ contain at most one variable, it turns out that they have one of the following forms:
(3.1) $t$ is a constant and contains $k$ $(0\leq k)$ occurrences of $-$, $s$ is a constant and contains $l$  $(0\leq l)$ occurrences of $-$. Then we have $t^{\scriptscriptstyle \textbf{F}}=0^{\scriptscriptstyle \textbf{F}} = s^{\scriptscriptstyle \textbf{F}}$ according to (2), so $t^{\scriptscriptstyle \textbf{F}} = s^{\scriptscriptstyle \textbf{F}}$, against the hypothesis.
(3.2) $t$ is a variable $x$ and contains $k$  $(0\leq k)$ occurrences of $-$, $s$ is a constant and contains $l$  $(0\leq l)$ occurrences of $-$. Then we have $s^{\textbf{F}(\textbf{MV*}_{[-1,1]},0)} = 0$ according to (2). We can assign $x$ the value $0 \neq a\in [-1,1]$, then we have $t^{\scriptscriptstyle \textbf{F}(\textbf{MV*}_{[-1,1]},0)}(a) \neq 0 = s^{\scriptscriptstyle \textbf{F}(\textbf{MV*}_{[-1,1]},0)}(a)$. This is a contradiction with $\textbf{F}(\textbf{MV*}_{[-1,1]},0) \models t \thickapprox s$.
(3.3) $t$ is a variable $x$ and contains $k$  $(0\leq k)$ occurrences of $-$, $s$ is a variable $y$ and contains $l$ $(0\leq l)$ occurrences of $-$. Then we can assign $x$ the value $0=a\in [-1,1]$ and $y$ the value $0\neq b \in [-1,1]$, so $t^{\textbf{F}(\textbf{MV*}_{[-1,1]},0)}(a) \neq s^{\textbf{F}(\textbf{MV*}_{[-1,1]},0)}(b)$. This is a contradiction with $\textbf{F}(\textbf{MV*}_{[-1,1]},0) \models t \thickapprox s$.
(3.4) $t$ is a variable $x$ and contains $k$  $(0\leq k)$ occurrences of $-$, $s$ is a variable $x$ and contains $l$  $(0\leq l)$ occurrences of $-$. If $k$ and $l$ have same parity, then $t^{\scriptscriptstyle \textbf{F}} = s^{\scriptscriptstyle \textbf{F}}$, against the hypothesis.
If $k$ and $l$ have different parity, then we can assign $x$ the value $0\neq a\in [-1,1]$, it follows that $t^{\scriptscriptstyle \textbf{F}(\textbf{MV*}_{[-1,1]},0)}(a) \neq s^{\scriptscriptstyle \textbf{F}(\textbf{MV*}_{[-1,1]},0)}(a)$.
This is a contradiction with $\textbf{F}(\textbf{MV*}_{[-1,1]},0) \models t \thickapprox s$.

Thus,  if $\textbf{F}(\textbf{MV*}_{[-1,1]},0) \models t \thickapprox s$, then we have $\mathbb{FSQMV^\ast} \models t \thickapprox s$.
\end{proof}

\begin{lem}\emph{\cite{11}}\label{03}
Let $t$ and $s$ be terms in the language of MV*-algebra. Then $\mathbb{MV^\ast} \models t \thickapprox s$ iff $\textbf{MV*}_{[-1,1]} \models t \thickapprox s$.
\end{lem}

\begin{thm}\label{05}
Let $t$ and $s$ be terms in the language of strong quasi-MV* algebra. Then we have $\mathbb{SQMV^\ast} \models t \thickapprox s$ iff $\textbf{S*} \models t \thickapprox s$.
\end{thm}

\begin{proof}
If $\mathbb{SQMV^\ast} \models t \thickapprox s$, then we have $\textbf{S*} \models t \thickapprox s$.

Conversely, if $\textbf{S*} \models t \thickapprox s$ and we suppose that $\mathbb{SQMV^\ast} \nvDash t \thickapprox s$, then there exists a strong quasi-MV* algebra \textbf{Q} such that $\textbf{Q} \nvDash t\thickapprox s$. By Proposition \ref{02}, we know that \textbf{Q} can be embedded into the direct product $\textbf{Q}/\mu \times \textbf{Q}/\tau$. Since $\textbf{Q} \nvDash t\thickapprox s$, we have either $\textbf{Q}/\mu \nvDash t\thickapprox s$ or $\textbf{Q}/\tau \nvDash t\thickapprox s$. If the former holds, then we have  $\textbf{MV*}_{[-1,1]} \nvDash t \thickapprox s$ by Lemma \ref{03}, so $\textbf{S*} \nvDash t \thickapprox s$. If the latter holds, then we have $\textbf{F}(\textbf{MV*}_{[-1,1]},0) \nvDash t \thickapprox s$ by Proposition \ref{04}, so $\textbf{S*} \nvDash t \thickapprox s$.

Thus, if $\textbf{S*} \models t \thickapprox s$,  then we have $\mathbb{SQMV^\ast} \models t \thickapprox s$.
\end{proof}

Now we prove that the square standard strong quasi-MV* algebra $\textbf{S*}$ and the disk standard strong quasi-MV* algebra $\textbf{D*}$ have the same equational theory.

\begin{lem}\label{L1}
Let $t(x_1, \cdots , x_n)$ be a term in the language of strong quasi-MV* algebra and contain at least an occurrence of $\oplus$. Then for any $\langle a_1,b_1 \rangle, \cdots , \langle a_n,b_n \rangle \in S^{*}$, we have $t(\langle a_1,b_1 \rangle, \cdots ,$ $ \langle a_n,b_n \rangle)=t(\langle a_1, 0 \rangle, \cdots , \langle a_n, 0 \rangle)=\langle z,0 \rangle$ for some $z\in [-1,1]$.
\end{lem}

\begin{proof}
We prove the result by induction on the number of occurrences of $\oplus$.

If $t$ contains one occurrence of $\oplus$, then $t= (-)^{r} (p \oplus q)$ where $p$, $q$ are terms which don't contain $\oplus$ and $(-)^{r}$ denotes $r (0\leq r)$ occurrences of $-$. We distinguish several cases in order to discuss.

Suppose that $p$ is a variable $x$ with $k(0\leq k)$ occurrences of $-$ and $q$ is a variable $y$ with $l(0\leq l)$ occurrences of $-$. Then $t= (-)^{r}((-)^{k}x \oplus (-)^{l}y)$. If $r$, $k$ and $l$ are even, then for any $\langle a,b \rangle, \langle c,d \rangle \in S^*$, we have $t(\langle a,b \rangle, \langle c,d \rangle)=\langle a,b \rangle \oplus \langle c,d \rangle = \langle  \max\{-1, \min\{ 1, a+c\}\}, 0 \rangle =\langle a,0 \rangle \oplus \langle c,0 \rangle$, it turns out that $t(\langle a,b \rangle, \langle c,d \rangle)=t(\langle a,0 \rangle, \langle c,0 \rangle)= \langle z,0 \rangle $ where $z=\max\{-1, \min\{ 1, a+c\}\} \in [-1,1]$.
If $k$, $l$ are even and $r$ is odd, then for any $\langle a,b \rangle, \langle c,d \rangle \in S^*$, we have $t(\langle a,b \rangle, \langle c,d \rangle)=-(\langle a,b \rangle \oplus \langle c,d \rangle) = \langle  \max\{-1, \min\{ 1, -a-c\}\}, 0 \rangle =-(\langle a,0 \rangle \oplus \langle c,0 \rangle)$, it turns out that $t(\langle a,b \rangle, \langle c,d \rangle)=t(\langle a,0 \rangle, \langle c,0 \rangle)= \langle z,0 \rangle $ where $z=\max\{-1, \min\{ 1, -a-c\}\} \in [-1,1]$.
 The rest can be proved similarly.

Suppose that $p$ is a variable $x$ with $k(0\leq k)$ occurrences of $-$ and $q$ is a constant with $l(0\leq l)$ occurrences of $-$. Then $t =(-)^{r}((-)^{k}x \oplus (-)^{l}0)$ or $t =(-)^{r}((-)^{k}x \oplus (-)^{l}1)$. If the former holds, without loss of generality, we can assume that $r$, $k$ and $l$ are even, then for any $\langle a,b \rangle \in S^*$, we have $t(\langle a,b \rangle, \langle 0,0 \rangle)= \langle a,b \rangle \oplus \langle 0,0 \rangle =\langle \max\{-1, \min\{ 1, a+0\}\}, 0 \rangle = \langle a,0 \rangle=\langle a,0 \rangle \oplus \langle 0,0 \rangle$, it means that $t(\langle a,b \rangle , \langle 0,0 \rangle)= t(\langle a,0 \rangle , \langle 0,0 \rangle) = \langle z, 0 \rangle$, where $z=a \in [-1,1]$. The rest can be proved similarly.
If the later holds, without loss of generality, we can assume that $r$, $k$ and $l$ are odd, then for any $\langle a,b \rangle \in S^*$, we have $t(\langle a,b \rangle, \langle 1,0 \rangle)= -(-\langle a,b \rangle \oplus (-\langle 1,0 \rangle)) =\langle \max\{-1, \min\{ 1, a+1\}\}, 0 \rangle = -(-\langle a,0 \rangle \oplus (-\langle 1,0 \rangle))$, it means that $t(\langle a,b \rangle, \langle 1,0 \rangle)=t(\langle a,0 \rangle, \langle 1,0 \rangle) = \langle z,0 \rangle$ where $z=\max\{-1, \min\{ 1, a+1\}\} \in [-1,1]$. The rest can be proved similarly.

Suppose that $p$ is a constant with $k (0\leq k)$ occurrences of $-$ and $q$ is a constant with $l (0\leq l)$ occurrences of $-$. Then $p = (-)^{k}0$ or $(-)^{k}1$, $q = (-)^{l}0$ or $(-)^{l}1$. We only prove when $p = (-)^{k}1$, $q = (-)^{l}1$ and $k$, $l$ are even. The rest can be proved similarly. For any $\langle a_1,b_1 \rangle, \cdots , \langle a_n,b_n \rangle \in S^*$, $t(\langle a_1,b_1 \rangle, \cdots , \langle a_n,b_n \rangle)= (-)^{r}(\langle 1,0 \rangle \oplus \langle1,0 \rangle) = t(\langle a_1, 0 \rangle, \cdots , \langle a_n, 0 \rangle)=\langle z,0 \rangle$, where $z = (-)^{r}1 \in [-1,1]$.

Now we assume that $t(\langle a_1,b_1 \rangle, \cdots, \langle a_n, b_n \rangle)=t(\langle a_1, 0 \rangle, \cdots, \langle a_n, 0 \rangle)=\langle z, 0 \rangle$ holds when $t$ contains $m (m\geq 2)$ occurrences of $\oplus$. Then, when $t$ contains $m+1$ occurrences of $\oplus$, we suppose that $t=(-)^{r} (p \oplus q)$ where $p$ and $q$ are terms containing the occurrences of $\oplus$ no more than $m$. If $r$ is even, then for any $\langle a_1,b_1 \rangle, \cdots , \langle a_n,b_n \rangle \in S^*$, we have
$$\begin{aligned}
t(\langle a_1,b_1 \rangle, \cdots, \langle a_n,b_n \rangle) &=p(\langle a_1,b_1 \rangle, \cdots, \langle a_n,b_n \rangle) \oplus q(\langle a_1,b_1 \rangle, \cdots , \langle a_n,b_n \rangle)\\
&=p(\langle a_1,0 \rangle, \cdots, \langle a_n,0 \rangle) \oplus q(\langle a_1,0 \rangle, \cdots, \langle a_n,0 \rangle)\\
&=t(\langle a_1, 0 \rangle, \cdots, \langle a_n, 0 \rangle),\\
\end{aligned}$$
and
$$\begin{aligned}
t(\langle a_1, 0 \rangle, \cdots , \langle a_n, 0 \rangle) &=p(\langle a_1,0 \rangle, \cdots , \langle a_n,0 \rangle) \oplus q(\langle a_1,0 \rangle, \cdots, \langle a_n,0 \rangle)\\
 &= \langle z_1,0 \rangle \oplus \langle z_2, 0 \rangle \\
 &= \langle \max\{ -1, \min\{ 1, z_1 + z_2 \} \} ,0 \rangle.\\
\end{aligned}$$
Put $z = \max\{ -1, \min\{ 1, z_1 + z_2 \} \}$, then we have $z\in [-1,1]$ and $t(\langle a_1,b_1 \rangle, \cdots, \langle a_n,b_n \rangle) = t(\langle a_1, 0 \rangle,$ $ \cdots, \langle a_n, 0 \rangle) = \langle z,0 \rangle$. If $r$ is odd, then for any $\langle a_1,b_1 \rangle, \cdots , \langle a_n,b_n \rangle \in S^{*}$, we can also show that the result is true.

Hence $t(\langle a_1,b_1 \rangle, \cdots, \langle a_n,b_n \rangle)=t(\langle a_1, 0 \rangle, \cdots, \langle a_n, 0 \rangle)= \langle z,0 \rangle$ for some $z \in [-1,1]$.
\end{proof}

\begin{prop}\label{06}
Let $t$ and $s$ be terms in the language of strong quasi-MV* algebra. Then we have $\textbf{S*} \models t \thickapprox s$ iff $\textbf{D*} \models t \thickapprox s$.
\end{prop}

\begin{proof}
If $\textbf{S*} \models t \approx s$, then we have $\textbf{D*} \models t \approx s$.

If $\textbf{D*} \models t \approx s$ and assume that $\textbf{S*} \nvDash t \approx s$, then there exist $\langle a_1,b_1 \rangle, \cdots, \langle a_n,b_n \rangle \in S^*$ such that $t^{\scriptscriptstyle \textbf{S*}}(\langle a_1 , b_1 \rangle, \cdots , \langle a_n, b_n \rangle) \neq s^{\scriptscriptstyle \textbf{S*}}(\langle a_1 , b_1 \rangle, \cdots , \langle a_n, b_n \rangle)$. We distinguish several cases in order to discuss.

(1) Both $t$ and $s$ contain at least an occurrence of $\oplus$. Since $t^{\scriptscriptstyle \textbf{S*}}(\langle a_1 , b_1 \rangle, \cdots , \langle a_n, b_n \rangle)$ $ \neq s^{\scriptscriptstyle \textbf{S*}}(\langle a_1 , b_1 \rangle,$ $ \cdots , \langle a_n, b_n \rangle)$, we have that $t^{\scriptscriptstyle \textbf{S*}}(\langle a_1 , 0 \rangle, \cdots , \langle a_n, 0 \rangle)=t^{\scriptscriptstyle \textbf{S*}}(\langle a_1 , b_1 \rangle, \cdots ,$ $ \langle a_n, b_n \rangle) \neq s^{\scriptscriptstyle \textbf{S*}}(\langle a_1 , b_1 \rangle, \cdots , $ $\langle a_n, b_n \rangle) = s^{\scriptscriptstyle \textbf{S*}}(\langle a_1 , 0 \rangle, \cdots , \langle a_n, 0 \rangle)$ from Lemma \ref{L1}. Because $\langle a_1, 0 \rangle, \cdots , \langle a_n,$ $ 0 \rangle \in D^*$, we have $t^{\scriptscriptstyle \textbf{D*}}(\langle a_1 , 0 \rangle, \cdots , \langle a_n, 0 \rangle) \neq $ $ s^{\scriptscriptstyle \textbf{D*}}(\langle a_1 , 0 \rangle, \cdots , \langle a_n, 0 \rangle)$. This is a contradiction with $\textbf{D*} \models t \approx s$.

(2) Either $t$ or $s$ contains at least an occurrence of $\oplus$. Without loss of generality, we assume that $t$ contains at least an occurrence of $\oplus$ and $s$ does not contain an occurrence of $\oplus$.
Then $s$ is one of the following forms: $s$ is a variable $x_i$ and contains $k$ $(0 \leq k)$ occurrences of $-$, or $s$ is a constant contains $k$ $(0 \leq k)$ occurrences of $-$.
 If $s$ is the former, we can assign $x_i$ the value $\langle c_i,d_i \rangle \in D^*$ where $d_i \neq 0$ and the variables in $\{ x_1, \cdots, x_{i-1}, x_{i+1}, \cdots, x_n \}$ any values $\langle c_j,d_j \rangle(j=1,\cdots,i-1,i+1,\cdots,n) \in D^*$, then we have $t^{\scriptscriptstyle \textbf{D*}}(\langle c_1 , d_1 \rangle, \cdots , \langle c_n, d_n \rangle)$ $ = t^{\scriptscriptstyle \textbf{S*}}(\langle c_1 , d_1 \rangle, \cdots , \langle c_n, d_n \rangle) = t^{\scriptscriptstyle \textbf{S*}}(\langle c_1 , 0 \rangle, \cdots , \langle c_n, 0 \rangle) = \langle z,0 \rangle$ for some $z \in [-1,1]$ from Lemma \ref{L1}, while $s^{\scriptscriptstyle \textbf{D*}}(\langle c_1 , d_1 \rangle, \cdots , \langle c_n, d_n \rangle) = \langle c_i,d_i \rangle$ or $\langle -c_i, -d_i \rangle$. Note that $d_i \neq 0$, so $t^{\scriptscriptstyle \textbf{D*}}(\langle c_1 , d_1 \rangle,$ $ \cdots, \langle c_n, d_n \rangle) \neq s^{\scriptscriptstyle \textbf{D*}}(\langle c_1 , d_1 \rangle, \cdots,$ $ \langle c_n, d_n \rangle)$. This is a contradiction with $\textbf{D*} \models t \approx s$.
 If $s$ is the later, since all constants of $\textbf{S*}$ belong to $\textbf{D*}$, we have $s^{\scriptscriptstyle \textbf{D*}}=s^{\scriptscriptstyle \textbf{S*}}$. By the hypothesis and Lemma \ref{L1}, we have that $t^{\scriptscriptstyle \textbf{S*}}(\langle a_1 , 0 \rangle, \cdots , \langle a_n, 0 \rangle)=t^{\scriptscriptstyle \textbf{S*}}(\langle a_1 , b_1 \rangle, \cdots , \langle a_n, b_n \rangle) \neq s^{\scriptscriptstyle \textbf{S*}}(\langle a_1 , b_1 \rangle, \cdots , \langle a_n, b_n \rangle) = s^{\scriptscriptstyle \textbf{S*}}(\langle a_1 , 0 \rangle, \cdots , \langle a_n, 0 \rangle)$. Since $\langle a_1 , 0 \rangle, \cdots , \langle a_n, 0 \rangle \in D^*$, we have $t^{\scriptscriptstyle \textbf{D*}}(\langle a_1 , $ $0 \rangle, \cdots , \langle a_n, 0 \rangle) $ $\neq s^{\scriptscriptstyle \textbf{D*}}(\langle a_1 , 0 \rangle, \cdots , \langle a_n, 0 \rangle)$. This is a contradiction with $\textbf{D*} \models t \approx s$.

(3) Neither $t$ nor $s$ contains any occurrence of $\oplus$, then $t$ and $s$ contain at most one variable, it turns out that they have one of the following forms:
(3.1) $t$ is a constant and contains $k$ $(0\leq k)$ occurrences of $-$, $s$ is a constant and contains $l$ $(0\leq l)$ occurrences of $-$. Since all constants of $\textbf{S*}$ belong to $\textbf{D*}$, we have $t^{\scriptscriptstyle \textbf{D*}} = t^{\scriptscriptstyle \textbf{S*}}$ and $s^{\scriptscriptstyle \textbf{D*}}= s^{\scriptscriptstyle \textbf{S*}}$. So if $t^{\scriptscriptstyle \textbf{S*}} \neq s^{\scriptscriptstyle \textbf{S*}}$, we have $t^{\scriptscriptstyle \textbf{D*}} \neq s^{\scriptscriptstyle \textbf{D*}}$ which is a contradiction with $\textbf{D*} \models t \approx s$.
(3.2) $t$ is a variable $x$ and contains $k$ $(0\leq k)$ occurrences of $-$, $s$ is a constant and contains $l$ $(0\leq l)$ occurrences of $-$, then $s^{\scriptscriptstyle \textbf{D*}} = s^{\scriptscriptstyle \textbf{S*}} = (-)^{l}0$ or $(-)^{l}1$. We assign $x$ the value $\langle c,d \rangle \in D^*$ where $d \neq 0$, then we have $t^{\scriptscriptstyle \textbf{D*}}(\langle c,d \rangle) \neq s^{\scriptscriptstyle \textbf{D*}}(\langle c,d \rangle)$. This is a contradiction with $\textbf{D*} \models t \approx s$.
(3.3) $t$ is a variable $x$ and contains $k$ $(0\leq k)$ occurrences of $-$, $s$ is a variable $y$ and contains $l$ $(0\leq l)$ occurrences of $-$. Then we can assign $x$ the value $\langle 0,0 \rangle$ and assign $y$ the value $\langle 0,0 \rangle \neq \langle c,d \rangle \in D^*$, so $t^{\scriptscriptstyle \textbf{D*}}(\langle 0,0 \rangle) \neq s^{\scriptscriptstyle \textbf{D*}}( \langle c,d \rangle)$. This is a contradiction with $\textbf{D*} \models t \approx s$.
(3.4) $t$ is a variable $x$ and contains $k$ $(0\leq k)$ occurrences of $-$, $s$ is a variable $x$ and contains $l$ $(0\leq l)$ occurrences of $-$. If $k$ and $l$ have same parity, then $t^{\scriptscriptstyle \textbf{S*}} = s^{\scriptscriptstyle \textbf{S*}}$, against the hypothesis. If $k$ and $l$ have different parity, we can assign $x$ the value $\langle 0,0 \rangle \neq \langle c,d \rangle \in D^*$, it turns out that $t^{\scriptscriptstyle \textbf{D*}}(\langle c,d \rangle) \neq s^{\scriptscriptstyle \textbf{D*}}(\langle c,d \rangle)$.  This is a contradiction with $\textbf{D*} \models t \approx s$.

Thus, if $\textbf{D*} \models t \approx s$, then we have $\textbf{S*} \models t \approx s$.
\end{proof}

\begin{coro}
Let $t$ and $s$ be terms in the language of strong quasi-MV* algebras. Then we have that $\mathbb{SQMV}^\ast \models t \thickapprox s$ iff $\textbf{D*} \models t \thickapprox s$.
\end{coro}

\begin{proof}
Follows from Theorem \ref{05} and Proposition \ref{06}.
\end{proof}

In order to investigate the logic associated with strong quasi-MV* algebras, we introduce strong quasi-Wajsberg* algebras which are term equivalence to strong quasi-MV* algebras.

\begin{defn}
A quasi-Wajsberg* algebra $\textbf{S}=\langle S; \to, \neg, ^+, ^-, 1 \rangle$ is called a \emph{strong quasi-Wajsberg* algebra}, if it satisfies the equations $x^+ = (1\to 1)\to x^+ $ and $x^- = (1\to 1)\to x^- $ for any $x \in S$.
\end{defn}

In the following, we abbreviate a strong quasi-Wajsberg* algebra $\textbf{S}=\langle S; \to, \neg, ^+, ^-, 1 \rangle$ as $\textbf{S}$ and denote the variety of strong quasi-Wajsberg* algebras by $\mathbb{SQW^*}$.

\begin{rem}\label{remark2}
Let $\textbf{S}$ be a strong quasi-Wajsberg* algebra. Then for any $x \in S$, we can get that $x^+ = (x\to 1)\to 1$ and $x^- = (x\to \neg 1)\to \neg 1$ from (QW*5).
\end{rem}

\begin{prop}\label{p3.2}
Let $\mathbf{Q}$ be a strong quasi-MV* algebra. If for any $x,y\in Q$, we define $x\to y=-x\oplus y$ and $\neg x=-x$, then $f(\mathbf{Q})=\langle Q;\to,\neg,^{+}, ^{-}, 1 \rangle$ is a strong quasi-Wajsberg* algebra, where $x\vee y = ((x^{+}\to y^{+})^{+}\to(\neg x)^{-}) \to((y^{-}\to x^{-})^{-}\to x^{-})$.
\end{prop}

\begin{proof}
By Proposition \ref{corr1}, we have that $f(\mathbf{Q})=\langle Q;\to,\neg,^{+}, ^{-}, 1 \rangle$ is a quasi-Wajsberg* algebra. Since $\mathbf{Q}$ is a strong quasi-MV* algebra, we have for any $x\in Q$, $x^+ = 1 \oplus (-1 \oplus x) = \neg 1 \to (1\to x) = \neg(1\to x)\to 1=(x\to 1)\to 1 = (1\to 1)\to x^+$ and $x^- = -1\oplus (1\oplus x) = 1 \to (\neg 1 \to x) = (\neg \neg )1 \to (\neg 1 \to x)= \neg(\neg 1 \to x) \to (\neg 1) = (x\to \neg 1)\to (\neg 1) =(1\to 1)\to x^-$ by Remark \ref{remark1}, Proposition \ref{lox1}(1), (QW*7), (QW*8), and (QW*5). Hence $f(\mathbf{Q})=\langle Q;\to,\neg,^{+}, ^{-}, 1 \rangle$ is a strong quasi-Wajsberg* algebra.
\end{proof}

\begin{prop}\label{p3.3}
Let $\mathbf{S}$ be a strong quasi-Wajsberg* algebra. If for any $x,y\in S$, we define $0=x \to x$, $x\oplus y=\neg x\to y$ and $-x=\neg x$, then $g(\mathbf{S})=\langle S;\oplus,{-},^{+}, ^{-},0,1\rangle$ is a strong quasi-MV* algebra, where $x\vee y = (x^{+}\oplus(-x^{+} \oplus y^{+})^{+})\oplus(x^{-}\oplus (-x^{-}\oplus y^{-})^{+})$.
\end{prop}

\begin{proof}
By Proposition \ref{corr2}, we have that $g(\mathbf{S})=\langle S;\oplus,{-},^{+}, ^{-},0,1\rangle$ is a quasi-MV* algebra. Since $\mathbf{S}$ is a strong quasi-Wajsberg* algebra, we have for any $x\in S$, $x^+ = (x\to 1)\to 1 = -(-x\oplus 1)\oplus 1 =1\oplus (-1\oplus x)= x^+ \oplus 0$ and $x^- = (x\to \neg 1)\to \neg 1 = -(-x\oplus (-1))\oplus (-1) = -1 \oplus (1\oplus x) = x^- \oplus 0$ by Remark \ref{remark2}, (QMV*1), (QMV*9), and (QMV*5). Hence $g(\mathbf{S})=\langle S;\oplus,{-},^{+}, ^{-},0,1\rangle$ is a strong quasi-MV* algebra.
\end{proof}

\begin{thm}\label{th2}
For any strong quasi-MV* algebra $\textbf{Q}$ and strong quasi-Wajsberg* algebra $\textbf{S}$, we have

\emph{(1)} $gf(\textbf{Q})=\textbf{Q}$,

\emph{(2)} $fg(\textbf{S})=\textbf{S}$.

So $f$ and $g$ are mutually inverse correspondence.
\end{thm}

\begin{coro}
Let \textbf{S} be a strong quasi-Wajsberg* algebra. Then there exist a Wajsberg* algebra \textbf{M} and a flat strong quasi-Wajsbeg* algebra \textbf{FW*} such that \textbf{S} can be embedded into the direct product $\textbf{M} \times \textbf{FW*}$.
\end{coro}

\begin{proof} Follows from Theorem \ref{th2} and Proposition \ref{02}.
\end{proof}

\begin{exam}\label{exampleSW}
The algebra $\textbf{SW*} = \langle S^{*}, \to, \neg, ^+, ^-, \textbf{1} \rangle$ is a strong quasi-Wajsberg* algebra, where for any $\langle a,b \rangle, \langle c,d \rangle \in S^{*}$,

$\langle a,b \rangle \to \langle c,d \rangle = \langle \max\{-1, \min\{ 1, c-a\}\}, 0 \rangle$,

$\neg \langle a,b \rangle = \langle -a, -b \rangle$,

$\langle a,b \rangle ^+ = \langle \max\{ 0,a \}, 0 \rangle$,

$\langle a,b \rangle ^- = \langle \min\{ 0,a \}, 0 \rangle$,

$\textbf{1}= \langle 1,0 \rangle$.

We call $\textbf{SW*}$ the \emph{square standard strong quasi-Wajsberg* algebra}.
Especially,  $\textbf{DW*}=\langle D^*, \to_{\scriptscriptstyle DW^*}, \neg_{\scriptscriptstyle DW^*}, ^{+_{DW^*}}, ^{-_{DW^*}},$ $ \textbf{1} \rangle$ is the subalgebra of \textbf{SW*}, where the operations $\to_{\scriptscriptstyle DW^*}$, $\neg_{\scriptscriptstyle DW^*}$, $^{+_{DW^*}}$ and $^{-_{DW^*}}$ are those of $\textbf{SW*}$ restricted to $D^*$. We call $\textbf{DW*}$ the \emph{disk standard strong quasi-Wajsberg* algebra}.
\end{exam}

Based on the completeness of strong quasi-MV* algebras, we have that $\mathbb{SQW^*}\models t\approx s$ iff $\textbf{SW*}\models t\approx s$.

In the end of this section, we focus on the relationship between strong quasi-Wajsberg* algebras and Wajsberg* algebras.

\begin{defn}
Let $t$ be a term in the language of strong quasi-Wajsberg* algebra. Then $t$ is called \emph{regular}, if it contains $\to$ or $1$.
\end{defn}

We notice that non-regular terms are exactly the ones belonging to the set $\{ \neg^{n} p: p$ is a variable and $0 \leq n \}$.

\begin{prop}\label{logic4}
Let $t$ be a regular term and $r$ be an arbitrary term in the language of strong quasi-Wajsberg* algebra. Then the equation $(r\to r)\to t \approx t$ is valid in all strong quasi-Wajsberg* algebras.
\end{prop}

\begin{proof}
Let $t$ be a regular term.  If $t= \neg^n 1$ for some $0\leq n$, then we have $(r\to r)\to t \approx t$ from  Proposition 3.1 in \cite{8}.
If $t = \neg^n (t_1 \to t_2)$ where $t_1, t_2$ are terms and $0\leq n$, then we have $(r\to r)\to t \approx t$ by (QW*4) and Proposition \ref{lox1}(3).
Thus, if $t$ is regular, then $(r\to r)\to t \approx t$ is valid in all strong quasi-Wajsberg* algebras.
\end{proof}

\begin{prop}\label{logic2}
Let $t_1$ and $t_2$ be regular terms in the language of strong quasi-Wajsberg* algebra. Then $t_1\approx t_2$ holds in all strong quasi-Wajsberg* algebras iff it holds in all Wajsberg* algebras.
\end{prop}

\begin{proof}
Suppose that $t_1\approx t_2$ holds in all strong quasi-Wajsberg* algebras. Since any Wajsberg* algebra is a strong quasi-Wajsberg* algebra, we have that $t_1\approx t_2$ holds in all Wajsberg* algebras.
Conversely, suppose that $t_1\approx t_2$ holds in all Wajsberg* algebras. Since $t_1\approx t_2 \approx 1$ is valid in all flat strong quasi-Wajsberg* algebras, we have that $t_1\approx t_2$ holds in the square standard strong quasi-Wajsberg* algebra $\textbf{SW*}$, so $t_1\approx t_2$ holds in all strong quasi-Wajsberg* algebras.
\end{proof}

\begin{coro}\label{logic3}
An equation $t_1 \approx (t_2 \to 1)\to 1$ holds in all strong quasi-Wajsberg* algebras iff it holds in all Wajsberg* algebras.
\end{coro}

\begin{proof}
If $t_1$ is regular, since $(t_2 \to 1)\to 1$ is regular, the conclusion is immediate to get from Proposition \ref{logic2}.
If $t_1$ is non-regular, then $t_1=\neg^n p$ where $p$ is a variable and $0\leq n$. Since $\neg^n p \approx (t_2 \to 1)\to 1$ fails in some Wajsberg* algebras and also fails in some strong quasi-Wajsberg* algebras, we have that $\neg^n p \approx (t_2 \to 1)\to 1$ holds in all strong quasi-Wajsberg* algebras iff it holds in all Wajsberg* algebras.
Hence we have that $t_1 \approx (t_2 \to 1)\to 1$ holds in all strong quasi-Wajsberg* algebras iff it holds in all Wajsberg* algebras.
\end{proof}

\begin{defn}\label{de}
Let $t_1, \cdots, t_n,t$ be terms in the language of strong quasi-Wajsberg* algebra, and let $c_1, \cdots, c_n,c$ be constants (which are added if it is not already present) denoting the designated elements in any algebra $\textbf{A}\in \mathbb{SQW^*}$. Then we introduce the \emph{logic} $\vDash _{\scriptscriptstyle \mathbb{SQW^*}}$:

\begin{center}
$t_1, \cdots, t_n \vDash _{\scriptscriptstyle \mathbb{SQW^*}} t$ iff $\mathbb{SQW^*} \models t_1 \approx (c_1 \to 1)\to 1 \& \cdots \&t_n \approx (c_n \to 1)\to 1 $ $\Rightarrow t \approx (c \to 1)\to 1$.
\end{center}
\end{defn}

Following from Definition \ref{de}, we have that $q_1, \cdots, q_n \vDash_{\scriptscriptstyle \mathbb{W^*}} q$ iff $\mathbb{W^*}\models q_1 \approx (c_1 \to 1)\to 1 \&\cdots \& q_n \approx (c_n \to 1)\to 1 \Rightarrow q \approx (c\to 1)\to 1$.
So, if $q_1, \cdots, q_n \vdash_{\scriptscriptstyle L^*} q$, then $q_1, \cdots, q_n \vDash_{\scriptscriptstyle \mathbb{W^*}} q$. Moreover, according to the completeness and soundness of $\L^*$ in \cite{2} and \cite{13}, the converse is also true, i.e., if $q_1, \cdots, q_n\vDash_{\scriptscriptstyle \mathbb{W^*}} q$, then $q_1, \cdots, q_n\vdash_{\scriptscriptstyle L^*} q$.

\begin{lem}\label{logic5}
Let $q_1,\cdots, q_n, q, p$ be terms in the language of strong quasi-Wajsberg* algebra. Then we have

 \emph{(1)} if $q_1,\cdots, q_n \vDash_{\scriptscriptstyle \mathbb{SQW}^{*}} q$, then $q_1,\cdots, q_n \vDash_{\scriptscriptstyle \mathbb{W}^{*}} q$,

 \emph{(2)} $q_1,\cdots, q_n \vDash_{\scriptscriptstyle \mathbb{W}^{*}} q$ iff $q_1,\cdots, q_n \vDash_{\scriptscriptstyle \mathbb{SQW}^{*}} (p \to p)\to q$,

 \emph{(3)} if $q$ is regular, then $q_1,\cdots, q_n \vDash_{\scriptscriptstyle \mathbb{W}^{*}} q$ iff $q_1,\cdots, q_n \vDash_{\scriptscriptstyle \mathbb{SQW}^{*}} q$.

\end{lem}

\begin{proof}
(1) Suppose that $q_1,\cdots, q_n \vDash_{\scriptscriptstyle \mathbb{SQW}^{*}} q$. Since every Wajsberg* algebra is a strong quasi-Wajsberg* algebra, we have $q_1,\cdots, q_n \vDash_{\scriptscriptstyle \mathbb{W}^{*}} q$.

(2) Suppose that $q_1,\cdots, q_n \vDash_{\scriptscriptstyle \mathbb{SQW}^{*}} (p \to p)\to q$. Then we have $q_1,\cdots, q_n \vDash_{\scriptscriptstyle \mathbb{W}^{*}} (p \to p)\to q$ by (1). Since $(p \to p)\to q \approx q$ holds in all Wajsberg* algebras, we get that $q_1,\cdots, q_n \vDash_{\scriptscriptstyle \mathbb{W}^{*}} q$. Conversely, suppose that $q_1,\cdots, q_n \vDash_{\scriptscriptstyle \mathbb{W}^{*}} q$.  Since the subalgebra of regular elements of any strong quasi-Wajsberg* algebra is a Wajsberg* algebra, we have that $(p \to p)\to q \approx (c \to 1)\to 1$ is valid in all strong quasi-Wajsberg* algebras. So $ \mathbb{SQW}^{*}\models q_1 \approx (c_1 \to 1)\to 1 \& \cdots \&q_n \approx (c_n \to 1)\to 1 \Rightarrow (p \to p)\to q \approx (c \to 1)\to 1$, it turns out that $q_1,\cdots, q_n \vDash_{\scriptscriptstyle \mathbb{SQW}^{*}} (p \to p)\to q$.

(3) Follows from (2) and Proposition \ref{logic4}.
\end{proof}

\section{The logic associated with strong quasi-Wajsberg* algebras}

In this section, we study the logical system associated with strong quasi-Wajsberg* algebras and denote it by $\text{sq}\L^{*}$.

Let $V$ be a set of all propositional variables and $F(V)$ the formulas set generated by $V$ with logical connectives $\to$, $\neg$ and a constant $1$. Then $\langle F(V); \to, \neg,  1 \rangle$ is a free algebra.
For any $p, q\in F(V)$, the notation $p\leftrightarrow q$ stands for $p\to q$ and $q\to p$.

The axioms of $\text{sq}\L^{*}$ are defined as follows.

\noindent\textbf{Axioms schemas}

(Q1) $(p\to q)\leftrightarrow (\neg q\to \neg p)$,

(Q2) $1\leftrightarrow ((1\to p)\to 1)$,

(Q3) $p \leftrightarrow ((q\to q)\to p)$,

(Q4) $(p\to q)\leftrightarrow((q^{+}\to p^{-})\to (p^{+}\to q^{-}))$ where $p^+ = (p \to 1)\to 1 $ and $p^- = (p \to \neg 1)\to \neg 1 $,

(Q5) $\neg(p\to q)\leftrightarrow (q\to p)$,

(Q6) $(p\to (\neg p\to q))^+\leftrightarrow (p^+\to (\neg p^+\to q^+))$,

(Q7) $(p\to (q\vee r))\leftrightarrow((p\to r)\vee (p\to q))$ where $p\vee q = (((p^{+}\to q^{+})^{+})\to(\neg p)^{-} )\to((q^{-}\to p^{-})^{-}\to p^{-})$,

(Q8) $(p\vee (q\vee r))\leftrightarrow ((p\vee q)\vee r)$,

(Q9) $((p\to 1)\to((q\to 1)\to r))\to ((q\to 1)\to((p\to 1)\to r))$,

(Q10) $p\to 1$.

The deduction rules of $\text{sq}\L^{*}$ are as follows.

\noindent\textbf{Rules of deduction}

(qMP) $(r\to r) \to p, (r\to r)\to (p\to q)\vdash (r\to r)\to q$,

(Reg) $p\vdash (r\to r)\to p$,

(AReg1) $(r\to r)\to (p\to q) \vdash p\to q$,

(AReg2) $(r\to r)\to \neg(p\to q) \vdash \neg(p\to q)$,

(AReg3) $(r\to r)\to \neg 1 \vdash \neg 1$,

(AReg4) $(r\to r)\to 1 \vdash 1$,

(Inv1) $p \vdash \neg \neg p$,

(Inv2) $\neg \neg p \vdash p$,

(Flat) $p, \neg 1 \vdash \neg p$,

(R2$^{\prime}$) $p\to q, r\to t\vdash (q\to r)\to (p\to t)$,

(R3$^{\prime}$) $(r\to r) \to p\vdash p^-$.

\begin{defn}
Let $F(V)$ be the formulas set of $\text{sq}\L^{*}$. If $q_1,q_2,\cdots,q_n$ ($1\leq n$) is a sequence of formulas such that one of the following cases holds:

(1) $q_i(i\le n)$ is an axiom,

(2) $q_i(i\le n)$ is obtained by $q_j$ and $q_k$ for some $j,k<i$ with the Rules of deduction,

\noindent then the sequence $q_1,q_2,\cdots,q_n$ is called a \emph{proof} of $q_n$ and denoted by $\vdash_{\scriptscriptstyle sqL^*} q_n$, $q_n$ is called a \emph{theorem}. The set of theorems in \text{sq}$\L^{*}$ is denoted by \emph{Thm}$_{\scriptscriptstyle sqL^*}$.
\end{defn}

\begin{defn}
Let $\Gamma\subseteq F(V)$ be a subset of formulas of $\text{sq}\L^{*}$. If $q_1,q_2,\cdots,q_n$ ($1\leq n$) is a sequence of formulas such that one of the following cases holds:

(1) $q_i(i\le n)$ is an axiom,

(2) $q_i\in \Gamma(i\le n)$,

(3) $q_i(i\le n)$ is obtained by $q_j$ and $q_k$ for some $j,k<i$ with the Rules of deduction,

\noindent then the sequence $q_1,q_2,\cdots,q_n$ is called a \emph{proof from $\Gamma$ to $q_n$} and denoted by $\Gamma\vdash_{\scriptscriptstyle sqL^*} q_n$, $q_n$ is called the \emph{provable formula} from $\Gamma$. The set of all provable formulas from $\Gamma$ is denoted by ${\Gamma}^{\vdash}$.
\end{defn}

\begin{prop}\label{QL1}
Let $F(V)$ be the formulas set of \emph{sq$\L^{*}$}. Then the following hold for any $p, q, r, t\in F(V)$,

\emph{(1)} If $\vdash p\leftrightarrow q$, then $\vdash \neg p\leftrightarrow \neg q$,

\emph{(2)} If $\vdash p\leftrightarrow q$ and $\vdash r\leftrightarrow t$, then $\vdash (p\to r)\leftrightarrow(q\to t)$,

\emph{(3)} If $\vdash p\leftrightarrow q$ and $\vdash q\leftrightarrow r$, then $\vdash p\leftrightarrow r$,

\emph{(4)} $\vdash \neg (p\to q)\leftrightarrow (\neg p\to \neg q)$,

\emph{(5)} $\vdash p \rightarrow p$,

\emph{(6)} If $\vdash p_1\leftrightarrow r_1$, then $\vdash p\leftrightarrow r$, where $p_1$ is a subformula of $p$ and $r$ is obtained by replacing $p_1$ in $p$ with $r_1$,

\emph{(7)} $\vdash (p\to p)\leftrightarrow (q\to q)$,

\emph{(8)} $\vdash p \leftrightarrow \neg \neg p$,

\emph{(9)} $\vdash (\neg p\to q)\leftrightarrow (\neg q\to p)$,

\emph{(10)} $\vdash (\neg p)^{+}\leftrightarrow \neg p^{-}$ and $\vdash (\neg p)^{-}\leftrightarrow \neg p^{+}$,

\emph{(11)} $\vdash (p\vee q) \leftrightarrow (q \vee p)$.

\end{prop}

\begin{proof}
\begin{enumerate}
\renewcommand{\labelenumi}{(\theenumi)}
\item If $\vdash p\leftrightarrow q$, then
\begin{flalign*}
  1^{\circ} &\ p \to q &(\text{Hypothesis})\\
  2^{\circ} &\ (r\to r)\to (p\to q) & 1^{\circ}, (\text{Reg})\\
  3^{\circ} &\ (p\to q)\to (\neg q\to \neg p)& (\text{Q}1)\\
  4^{\circ} &\ (r\to r)\to((p\to q)\to (\neg q\to \neg p))& 1^{\circ}, (\text{Reg})\\
  5^{\circ} &\ (r\to r)\to(\neg q\to \neg p) & 2^{\circ}, 4^{\circ}, (\text{qMP})\\
  6^{\circ} &\ \neg q\to \neg p & 5^{\circ}, (\text{AReg1})
\end{flalign*}
By the similar way, we can get that $\vdash \neg p\to \neg q$. So $\vdash \neg p\leftrightarrow \neg q$.

\item If $\vdash p\leftrightarrow q$ and $\vdash r\leftrightarrow t$, then
\begin{flalign*}
  1^{\circ} & \ p\to q &(\text{Hypothesis})\\
  2^{\circ} & \ t\to r &(\text{Hypothesis})\\
  3^{\circ} & \ (q\to t)\to (p\to r)& 1^{\circ}, 2^{\circ}, (\text{R}2^{\prime})
\end{flalign*}
By the similar way, we can get that $\vdash (p\to r)\to(q\to t)$. So $\vdash (p\to r)\leftrightarrow(q\to t)$.

\item If $\vdash p\leftrightarrow q$ and $\vdash q\leftrightarrow r$, then
\begin{flalign*}
  1^{\circ} & \ p\to q &(\text{Hypothesis})\\
  2^{\circ} & \ q\to r &(\text{Hypothesis})\\
  3^{\circ} & \ (q\to q)\to (p\to r)& 1^{\circ}, 2^{\circ}, (\text{R}2^{\prime})\\
  4^{\circ} & \ (r\to r)\to ((q\to q)\to (p\to r))& 3^{\circ}, (\text{Reg})\\
  5^{\circ} & \ ((q\to q)\to (p\to r))\to (p\to r) & (\text{Q}3)\\
  6^{\circ} & \ (r\to r)\to (((q\to q)\to (p\to r))\to (p\to r)) & 5^{\circ}, (\text{Reg})\\
  7^{\circ} & \ (r\to r)\to (p\to r) & 4^{\circ}, 6^{\circ}, (\text{qMP})\\
  8^{\circ} & \ p\to r & 7^{\circ}, (\text{AReg1})
\end{flalign*}
Similarly, we can get that $\vdash r\to p$. So $\vdash p\leftrightarrow r$.

\item We have
\begin{flalign*}
  1^{\circ} & \ \neg (p\to q)\to ( q\to p) & (\text{Q}5)\\
  2^{\circ} & \ (q\to p)\to (\neg p\to \neg q) & (\text{Q}1)\\
  3^{\circ} & \ \neg (p\to q)\to (\neg p\to \neg q) & 1^{\circ}, 2^{\circ}, (3)
\end{flalign*}
Similarly, we can get that $\vdash (\neg p\to \neg q)\to \neg (p\to q)$. So $\vdash \neg (p\to q)\leftrightarrow (\neg p\to \neg q)$.

\item We have
\begin{flalign*}
  1^{\circ} & \ (p\to p) \to ((q\to q)\to (p\to p)) &(\text{Q}3)\\
  2^{\circ} & \ ((q\to q)\to (p\to p))\to (p\to p) & (\text{Q}3)\\
  3^{\circ} & \ (p\to p)\to (p\to p) & 1^{\circ}, 2^{\circ}, (3)\\
  4^{\circ} & \ p \to p & 3^{\circ}, (\text{AReg1})
\end{flalign*}
So $\vdash p\rightarrow p$.

\item Suppose that $\vdash p_1 \leftrightarrow r_1$. We use the induction for the number $n$ of connectives in $p$. If $n=0$, then $p=p_1$ and $r=r_1$, so we have $\vdash p \leftrightarrow r$. If $n=1$, then $p_1$ is a propositional variable, we have

(a) If $p=\neg p_1$, then $r=\neg r_1$.
\begin{flalign*}
  1^{\circ} & \ p_1\leftrightarrow r_1 &(\text{Hypothesis})\\
  2^{\circ} & \ \neg p_1\leftrightarrow \neg r_1 & 1^{\circ}, (1)
\end{flalign*}
So $\vdash p \leftrightarrow r$.

(b) If $p=p_1\to t_1$, then $r=r_1\to t_1$, where $t_1$ is a propositional variable.
\begin{flalign*}
  1^{\circ} & \ p_1\leftrightarrow r_1 &(\text{Hypothesis})\\
  2^{\circ} & \ t_1\leftrightarrow t_1 & (5)\\
  3^{\circ} & \ (p_1\to t_1)\leftrightarrow (r_1\to t_1) & 1^{\circ}, 2^{\circ}, (2)
\end{flalign*}

If $p= t_1 \to p_1$, then $r= t_1 \to r_1$, where $t_1$ is a propositional variable.
\begin{flalign*}
  1^{\circ} & \ t_1\leftrightarrow t_1 & (5)\\
  2^{\circ} & \ p_1\leftrightarrow r_1 & (\text{Hypothesis})\\
  3^{\circ} & \ (t_1\to p_1)\leftrightarrow (t_1\to r_1) & 1^{\circ}, 2^{\circ}, (2)
\end{flalign*}
So we have $\vdash p\leftrightarrow r$.

Assume that the result holds when $n\le k$. Now we verify that the following cases hold when $n=k+1$.

(c) If $p=\neg u$, and $p_1$ is a subformula of $u$, then we denote that $v$ is obtained by replacing $p_1$ with $r_1$ in $u$, so $r=\neg v$. Since the number of connectives in $u$ is not more than $k$ and $\vdash p_1\leftrightarrow r_1$, we have $\vdash u\leftrightarrow v$ by the inductive assumption. So we have $\vdash p\leftrightarrow r$.

(d) If $p=u\to q$ and $p_1$ is a subformula of $u$ or $q$, then we denote that $v$ or $t$ is obtained by replacing $p_1$ with $r_1$ in $u$ or $q$, so $r=v\to q$ or $r=u\to t$. Since the number of connectives in $u$ or $q$ is not more than $k$ and $\vdash p_1\leftrightarrow r_1$, we have $\vdash u\leftrightarrow v$ or $\vdash q\leftrightarrow t$ by the inductive assumption. So
\begin{flalign*}
  1^{\circ} & \ u\leftrightarrow v &(\text{Hypothesis})\\
  2^{\circ} & \ q\leftrightarrow q & (5)\\
  3^{\circ} & \ (u\to q)\leftrightarrow (v\to q) & 1^{\circ}, 2^{\circ}, (2)
\end{flalign*}
or
\begin{flalign*}
  1^{\circ} & \ u\leftrightarrow u & (5)\\
  2^{\circ} & \ q\leftrightarrow t&(\text{Hypothesis})\\
  3^{\circ} & \ (u\to q)\leftrightarrow (u\to t) & 1^{\circ}, 2^{\circ}, (2)
\end{flalign*}
Hence $\vdash p\leftrightarrow r$.

\item We have
\begin{flalign*}
  1^{\circ} & \ p\to p &(5)\\
  2^{\circ} & \ (r\to r)\to (p\to p) & 1^{\circ}, (\text{Reg})\\
  3^{\circ} & \ (p\to p)\to ((q\to q)\to (p\to p)) & (\text{Q}3)\\
  4^{\circ} & \ (r\to r)\to ((p\to p)\to ((q\to q)\to (p\to p)))& 3^{\circ}, (\text{Reg})\\
  5^{\circ} & \ (r\to r)\to ((q\to q)\to (p\to p)) & 2^{\circ}, 4^{\circ}, (\text{qMP})\\
  6^{\circ} & \ (q\to q)\to (p\to p) & 5^{\circ}, (\text{AReg}1)
\end{flalign*}
Swapping $p$ and $q$, we can get that $\vdash (p\to p)\leftrightarrow (q\to q)$.

\item We have
\begin{flalign*}
1^{\circ} & \ \neg(q\to q)\leftrightarrow (q\to q)& (\text{Q}5)\\
2^{\circ} & \ \neg\neg(q\to q)\leftrightarrow \neg(q\to q) & 1^{\circ}, (1)\\
  3^{\circ} & \ \neg\neg(q\to q)\leftrightarrow (q\to q) & 2^{\circ}, 1^{\circ}, (3)\\
  4^{\circ} & \ p\leftrightarrow ((q\to q)\to p)& (\text{Q}3)\\
  5^{\circ} & \ ((q\to q)\to p)\leftrightarrow (\neg p\to \neg(q\to q)) & (\text{Q}1)\\
  6^{\circ} & \ p\leftrightarrow (\neg p\to \neg(q\to q)) & 4^{\circ}, 5^{\circ}, (3)\\
  7^{\circ} & \ (\neg p\to \neg(q\to q))\leftrightarrow (\neg\neg(q\to q)\to \neg\neg p) & (\text{Q}1)\\
  8^{\circ} & \ p\leftrightarrow (\neg\neg(q\to q)\to \neg\neg p) & 6^{\circ}, 7^{\circ}, (3)\\
  9^{\circ} & \ p\leftrightarrow ((q\to q)\to \neg\neg p) & 8^{\circ}, 3^{\circ}, (6)\\
  10^{\circ} & \ ((q\to q)\to \neg\neg p)\leftrightarrow \neg\neg p & (\text{Q}3)\\
  11^{\circ} & \ p\leftrightarrow \neg\neg p & 9^{\circ}, 10^{\circ}, (3)
\end{flalign*}
So we have $\vdash p\leftrightarrow \neg\neg p$.

\item We have
\begin{flalign*}
  1^{\circ} & \ (\neg p\to q)\leftrightarrow (\neg q\to \neg\neg p) & (\text{Q}1)\\
  2^{\circ} & \ \neg\neg p\leftrightarrow p & (8)\\
  3^{\circ} & \ (\neg p\to q)\leftrightarrow (\neg q\to p) & 1^{\circ}, 2^{\circ}, (6)
\end{flalign*}
So $\vdash (\neg p\to q)\leftrightarrow (\neg q\to p)$.

\item Since $(\neg p)^+ = (\neg p\to 1)\to 1$ and $\neg p^- = \neg((p\to \neg 1)\to \neg 1)$, we have
\begin{flalign*}
  1^{\circ} & \ ((\neg p\to 1)\to 1)\leftrightarrow \neg(1\to (\neg p\to 1)) & (\text{Q}5)\\
  2^{\circ} & \ 1\leftrightarrow \neg\neg 1 & (8)\\
  3^{\circ} & \ ((\neg p\to 1)\to 1)\leftrightarrow \neg(1\to (\neg p\to \neg\neg 1)) & 1^{\circ}, 2^{\circ}, (6)\\
  4^{\circ} & \ (\neg p\to \neg \neg 1)\leftrightarrow \neg(p\to \neg 1) & (4)\\
  5^{\circ} & \ ((\neg p\to 1)\to 1) \leftrightarrow \neg(1\to \neg(p\to \neg 1)) & 3^{\circ}, 4^{\circ}, (6)\\
  6^{\circ} & \ ((\neg p\to 1)\to 1) \leftrightarrow \neg(\neg \neg 1\to \neg(p\to \neg 1)) & 2^{\circ}, 5^{\circ}, (6)\\
  7^{\circ} & \ (\neg \neg 1\to \neg(p\to \neg 1))\leftrightarrow ((p\to \neg 1)\to \neg 1) & (\text{Q}1)\\
  8^{\circ} & \ \neg(\neg \neg 1\to \neg(p\to \neg 1))\leftrightarrow \neg((p\to \neg 1)\to \neg 1) & 7^{\circ}, (1)\\
  9^{\circ} & \ ((\neg p\to 1)\to 1)\leftrightarrow \neg((p\to \neg 1)\to \neg 1) & 6^{\circ}, 8^{\circ}, (3)
\end{flalign*}
So $\vdash (\neg p)^{+}\leftrightarrow \neg p^-$. From the former result, we have $\vdash (\neg \neg p)^{+}\leftrightarrow \neg (\neg p)^-$, then
\begin{flalign*}
  1^{\circ} & \ (\neg \neg p)^{+}\leftrightarrow \neg(\neg p)^- & (\text{Hypothesis})\\
  2^{\circ} & \ \neg \neg p \leftrightarrow p& (8)\\
  3^{\circ} & \ p^{+}\leftrightarrow \neg(\neg p)^- & 1^{\circ}, 2^{\circ}, (6)\\
  4^{\circ} & \ \neg p^{+}\leftrightarrow \neg \neg(\neg p)^- & 3^{\circ}, (1)\\
  5^{\circ} & \ \neg \neg(\neg p)^-\leftrightarrow (\neg p)^- & (8)\\
  6^{\circ} & \ \neg p^{+}\leftrightarrow (\neg p)^- & 4^{\circ}, 5^{\circ}, (3)
\end{flalign*}
So we have $\vdash \neg p^{+}\leftrightarrow (\neg p)^-$.

\item We have
\begin{flalign*}
  1^{\circ} & \ ((r\to r)\to (p\vee q)) \leftrightarrow (((r\to r)\to q)\vee ((r\to r)\to p)) & (\text{Q}7)\\
  2^{\circ} & \ (p\vee q)\leftrightarrow ((r\to r)\to (p\vee q)) & (\text{Q}3)\\
  3^{\circ} & \ (p\vee q) \leftrightarrow (((r\to r)\to q)\vee ((r\to r)\to p)) & 2^{\circ}, 1^{\circ}, (3)\\
  4^{\circ} & \ ((r\to r)\to q)\leftrightarrow q & (\text{Q}3)\\
  5^{\circ} & \ (p\vee q) \leftrightarrow (q\vee ((r\to r)\to p))& 3^{\circ}, 4^{\circ}, (6)\\
  6^{\circ} & \ ((r\to r)\to p)\leftrightarrow p & (\text{Q}3)\\
  7^{\circ} & \ (p\vee q) \leftrightarrow (q\vee p)& 5^{\circ}, 6^{\circ}, (6)
\end{flalign*}

So $\vdash (p\vee q) \leftrightarrow (q \vee p)$.
\end{enumerate}
\end{proof}

In order to discuss the completeness and soundness of sq$\L^*$, we need some conclusions here.

\begin{lem}\label{lemma2}
Let $F(V)$ be the formulas set of \emph{sq$\L^{*}$} and $q_1, \cdots, q_n, q, p \in F(V)$. Then we have

\emph{(1)} if $q_1, \cdots, q_n \vdash_{\scriptscriptstyle L^{*}} q$, then $q_1, \cdots, q_n \vdash_{\scriptscriptstyle sqL^{*}} (p \to p)\to q$,

\emph{(2)} if $q$ is regular, then $(p \to p)\to q \vdash_{\scriptscriptstyle sqL^{*}} q$.
\end{lem}

\begin{proof}
(1) The proof is a simple induction on the length of the derivation of $q$ from $q_1, \cdots, q_n$ in $\L^{*}$.

If $q$ is the axiom of $\L^{*}$ or  $q \in \{ q_1, \cdots, q_n \}$, then by the rule (Reg), we have $q \vdash_{\scriptscriptstyle sqL^{*}} (p \to p)\to q$, so $q_1, \cdots, q_n \vdash_{\scriptscriptstyle sqL^{*}} (p \to p)\to q$.

If $q$ is derived from axioms and $\{ q_1, \cdots, q_n \}$ applying the Rules of deduction in $\L^*$, then we consider three subcases:

Suppose that $q$ is derived from the rule (R1). Then we can assume that $r, r \to q \vdash_{\scriptscriptstyle L^{*}} q$, where $r, r \to q \in \{$axioms$\} \cup \{ q_1, \cdots, q_n \}$. By the rule (Reg), we have $r \vdash_{\scriptscriptstyle sqL^{*}} (p \to p)\to r$ and $r \to q \vdash_{\scriptscriptstyle sqL^{*}} (p \to p)\to (r \to q)$, it follows that $(p \to p)\to r, (p \to p)\to (r \to q) \vdash_{\scriptscriptstyle sqL^{*}} (p \to p)\to q$ by the rule (qMP), so $q_1, \cdots, q_n \vdash_{\scriptscriptstyle sqL^{*}} (p \to p)\to q$.

Suppose that $q$ is derived from  the rule (R2). Then we can assume that $r \to t, u \to v \vdash_{\scriptscriptstyle L^{*}} (t \to u)\to (r \to v)$, where $r \to t, u \to v \in \{$axioms$\} \cup \{ q_1, \cdots, q_n \}$ and $q = (t \to u)\to (r \to v)$. By the rule (R2$^{\prime}$), we have $r \to t, u \to v \vdash_{\scriptscriptstyle sqL^{*}} q$ and then we get that $q \vdash_{\scriptscriptstyle sqL^{*}} (p \to p)\to q$  by the rule (Reg). So $q_1, \cdots, q_n \vdash_{\scriptscriptstyle sqL^{*}} (p \to p)\to q$.

Suppose that $q$ is derived from  the rule (R3). Then we can assume that $r \vdash_{\scriptscriptstyle L^{*}} r^{-}$, where $r \in \{$axioms$\} \cup \{ q_1, \cdots, q_n \}$ and $q = r^{-}$. By the rule (Reg), we have $r \vdash_{\scriptscriptstyle sqL^{*}} (p \to p)\to r$ and then we get $(p \to p)\to r \vdash_{\scriptscriptstyle sqL^{*}} r^{-}$ by the rule (R3$^{\prime}$). Applying the rule (Reg) again, we get $r^{-} \vdash_{\scriptscriptstyle sqL^{*}} (p \to p)\to r^{-}$. So $q_1, \cdots, q_n \vdash_{\scriptscriptstyle sqL^{*}} (p \to p)\to q$.

Thus, if $q_1, \cdots, q_n \vdash_{\scriptscriptstyle L^{*}} q$, then $q_1, \cdots, q_n \vdash_{\scriptscriptstyle sqL^{*}} (p \to p)\to q$.

(2) If $q = 1$, then we have $(p \to p)\to 1 \vdash_{\scriptscriptstyle sqL^{*}} 1$ by the rule (AReg4).

If $q = \neg^k 1$ $(1\leq k)$, we assume that $k$ is odd, then based on $(p \to p)\to \neg \neg r \vdash_{\scriptscriptstyle L^{*}} r$ in \cite{2}, we can deduce that $(p \to p)\to \neg^{k} 1 \vdash_{\scriptscriptstyle L^{*}} \neg 1$. By (1), we have $(p \to p)\to \neg^{k} 1 \vdash_{\scriptscriptstyle sqL^{*}} (p \to p)\to \neg 1$, and then we have $(p \to p)\to \neg 1 \vdash_{\scriptscriptstyle sqL^{*}} \neg 1$ from the rule (AReg3). Besides, we can deduce that $\neg 1 \vdash_{\scriptscriptstyle sqL^{*}} \neg^k 1$ by using the rule (Inv1) several times. Hence we have $(p \to p) \to q \vdash_{\scriptscriptstyle sqL^{*}} q$. If $k$ is even, then it can be proved similarly.

If $q = q_1 \to q_2$, where $q_1, q_2 \in F(V)$, then by the rule (AReg1), we have $(p \to p)\to q \vdash_{\scriptscriptstyle sqL^{*}} q$.

If $q = \neg^k (q_1 \to q_2)$ $(1\leq k)$, where $q_1, q_2 \in F(V)$. Assume that $k$ is odd. Then $(p \to p)\to \neg^{k} (q_1 \to q_2) \vdash_{\scriptscriptstyle L^{*}} \neg (q_1 \to q_2)$ from \cite{2} and then we have $(p \to p)\to \neg^{k} (q_1 \to q_2) \vdash_{\scriptscriptstyle sqL^{*}} (p \to p)\to \neg (q_1 \to q_2)$ from (1). Using the rule (AReg2), we get $(p \to p)\to \neg (q_1 \to q_2) \vdash_{\scriptscriptstyle sqL^{*}} \neg (q_1 \to q_2)$. And then we can deduce that $\neg (q_1 \to q_2) \vdash_{\scriptscriptstyle sqL^{*}} \neg^k (q_1 \to q_2)$ using the rule (Inv1) several times. So $(p \to p)\to q \vdash_{\scriptscriptstyle sqL^{*}} q$. If $k$ is even, then it can be proved similarly.

Therefore, if $q$ is regular, then $(p \to p)\to q \vdash_{\scriptscriptstyle sqL^{*}} q$.
\end{proof}

 Based on the fact that $\vDash_{\scriptscriptstyle \mathbb{W}^{*}} q$ iff $\vdash_{\scriptscriptstyle L^{*}} q$ in \cite{13}, we can show that sq$\L^{*}$ is coincided with $\vDash_{\scriptscriptstyle \mathbb{SQW^*}}$.

\begin{thm}\label{4.1}
$ \vDash_{\scriptscriptstyle \mathbb{SQW}^{*}} q$ iff $ \vdash_{\scriptscriptstyle sqL^{*}} q$.
\end{thm}

\begin{proof} For the one direction, we prove the completeness.

If $\vDash_{\scriptscriptstyle \mathbb{SQW}^{*}} q$, then $\mathbb{SQW}^{*} \models q \approx (c \to 1)\to 1$,  it turns out that $\mathbb{W}^{*} \models q \approx (c \to 1)\to 1$, i.e., $\vDash_{\scriptscriptstyle \mathbb{W}^{*}} q$. So we have $\vdash_{\scriptscriptstyle L^{*}} q$.
If $q$ is non-regular, then $q \approx (c \to 1)\to 1$ fails in some Wajsberg* algebras and then $\mathbb{W}^{*} \nvDash q \approx (c \to 1)\to 1$. Hence $q$ is regular and then we have $\vdash_{\scriptscriptstyle sqL^{*}} q$ by Lemma \ref{lemma2}(1) and Lemma \ref{lemma2}(2).

For the other direction, we prove the soundness.

If $\vdash_{\scriptscriptstyle sqL^{*}} q$, then we distinguish several cases in order to discuss.

(1) If $q$ is an axiom of sq$\L^{*}$, then $q$ is regular. Since the axioms of sq$\L^{*}$ are the same as $\L^{*}$, we have that $\vdash_{\scriptscriptstyle L^{*}} q$ and then $\vDash_{\scriptscriptstyle \mathbb{W}^{*}} q$. From Lemma \ref{logic5}(3), we have $\vDash_{\scriptscriptstyle \mathbb{SQW}^{*}} q$.

(2) If $q$ is a theorem of sq$\L^*$, then we need to verify that all the rules satisfy $\vDash_{\scriptscriptstyle \mathbb{SQW}^{*}}$.

(qMP) Suppose that $\mathbb{SQW}^{*} \models (r\to r)\to p \approx (c_1\to 1)\to 1 $ and $\mathbb{SQW}^{*} \models (r\to r)\to (p\to q) \approx (c_2\to 1)\to 1$. Then we have $\mathbb{W}^{*} \models (r\to r)\to p \approx (c_1\to 1)\to 1 $ and $\mathbb{W}^{*} \models (r\to r)\to (p\to q) \approx (c_2\to 1)\to 1$. Since $(r\to r)\to p \approx p $ and $(r\to r)\to (p\to q) \approx p\to q$ hold in any Wajsberg* algebra, we have $\mathbb{W}^{*} \models p \approx (c_1\to 1)\to 1 $ and $\mathbb{W}^{*} \models p\to q \approx (c_2\to 1)\to 1$. Applying the soundness of $\L^*$, we get $\mathbb{W}^{*} \models p \approx (c_1\to 1)\to 1 \& p\to q \approx (c_2\to 1)\to 1 \Rightarrow q \approx (c\to 1)\to 1$, it turns out that $\mathbb{W}^{*} \models (r\to r)\to p \approx (c_1\to 1)\to 1 \& (r\to r)\to (p\to q) \approx (c_2\to 1)\to 1 \Rightarrow q \approx (c\to 1)\to 1$, i.e., $(r\to r)\to p, (r\to r)\to (p\to q) \vDash_{\scriptscriptstyle \mathbb{W}^{*}} q$. So we have $(r\to r)\to p, (r\to r)\to (p\to q) \vDash_{\scriptscriptstyle \mathbb{SQW}^{*}} (r\to r)\to q$ by Lemma \ref{logic5}(2).

(Reg) Suppose that $\mathbb{SQW}^* \models p \approx (c\to 1)\to 1$. Since $(r\to r)\to p \approx (r\to r)\to ((c\to 1)\to 1)$ holds in any strong quasi-Wajsberg* algebra and $(c\to 1)\to 1$ is regular, we have that $(r\to r)\to ((c\to 1)\to 1)\approx (c\to 1)\to 1$, it turns out that $\mathbb{SQW}^* \models (r\to r)\to p \approx (c\to 1)\to 1$, i.e., $p \vDash_{\scriptscriptstyle \mathbb{SQW}^{*}} (r\to r)\to p$.

(AReg1) Suppose that $\mathbb{SQW}^* \models (r\to r)\to (p\to q) \approx (c\to 1)\to 1$. Since $p\to q$ is regular, we have that $(r\to r)\to (p\to q) \approx p\to q$ holds in any strong quasi-Wajsberg* algebra from Proposition \ref{logic4}. So $\mathbb{SQW}^* \models p\to q \approx (c\to 1)\to 1$, i.e., $(r\to r)\to (p\to q) \vDash_{\scriptscriptstyle \mathbb{SQW}^{*}} p\to q$.

Note that $\neg(p\to q)$, $\neg 1$ and $1$ are regular, the proofs of (AReg2)--(AReg4) are similar to (AReg1).

(Inv1)--(Inv2) Since $p\approx \neg \neg p$ holds in any strong quasi-Wajsberg* algebra, we have that $p\vDash_{\scriptscriptstyle \mathbb{SQW}^{*}} \neg \neg p$ and $\neg \neg p \vDash_{\scriptscriptstyle \mathbb{SQW}^{*}} p$.

(Flat) Suppose that $\mathbb{SQW}^* \models p\approx (c_1\to 1)\to 1 \& \neg1\approx (c_2 \to 1)\to 1$. For any strong quasi-Wajsberg* algebra, since $1=\neg \neg 1=\neg ((c_2 \to 1)\to 1)=1\to (c_2\to 1)$, we have $1\to 1 = (1\to (c_2 \to 1))\to 1 = 1$ and then $1\to 1=1\to (1\to 1)=\neg 1$, so $1=\neg 1$. Moreover, we have $\neg p = \neg((c_1\to 1)\to 1)=(\neg c_1 \to \neg 1)\to \neg 1= (\neg c_1 \to 1)\to 1$. It follows that $\mathbb{SQW}^* \models p\approx (c_1\to 1)\to 1 \& \neg1\approx (c_2 \to 1)\to 1 \Rightarrow \neg p \approx (\neg c_1 \to 1)\to 1$, i.e., $p, \neg 1 \vDash_{\scriptscriptstyle \mathbb{SQW}^{*}} \neg p$.

(R2$^{\prime}$) Because $p\to q, r\to t \vDash_{\scriptscriptstyle \mathbb{W}^{*}} (q\to r)\to (p\to t)$, we have $p\to q, r\to t \vDash_{\scriptscriptstyle \mathbb{SQW}^{*}} (u\to u)\to ((q\to r)\to (p\to t))$ from Lemma \ref{logic5}(2). Since $(u\to u)\to ((q\to r)\to (p\to t)) \approx (q\to r)\to (p\to t)$ holds in any strong quasi-Wajsberg* algebra, we have $p\to q, r\to t \vDash_{\scriptscriptstyle \mathbb{SQW}^{*}} (q\to r)\to (p\to t)$.

(R3$^{\prime}$) Suppose that $\mathbb{SQW}^* \models (r\to r)\to p\approx (c_1\to 1)\to 1$. Then we have $\mathbb{W}^* \models (r\to r)\to p\approx (c_1\to 1)\to 1$. Since $(r\to r)\to p\approx p$ holds in any strong quasi-Wajsberg* algebra, we have $\mathbb{W}^* \models p\approx (c_1\to 1)\to 1$. Applying the soundness of $\L^*$, we get $\mathbb{W}^* \models p\approx (c_1\to 1)\to 1\Rightarrow p^- \approx (c_2 \to 1)\to 1$, it turns out that $\mathbb{W}^* \models (r\to r)\to p\approx (c_1\to 1)\to 1\Rightarrow p^- \approx (c_2 \to 1)\to 1$, i.e., $(r\to r)\to p \vDash_{\scriptscriptstyle \mathbb{W}^{*}} p^-$. Note that $p^-$ is regular, we have $(r\to r)\to p \vDash_{\scriptscriptstyle \mathbb{SQW}^{*}} p^-$ from Lemma \ref{logic5}(3).

Therefore, $\vDash_{\scriptscriptstyle \mathbb{SQW}^{*}} q$ iff $\vdash_{\scriptscriptstyle sqL^{*}} q$.
\end{proof}

In the logic $\L^*$, authors have also showed that $q_1, \cdots, q_n \vdash_{\scriptscriptstyle L^{*}} q$ iff $q_1, \cdots, q_n\vDash_{\scriptscriptstyle \mathbb{W}^{*}} q$ \cite{13}.  For the logic sq$\L^*$, we can show the following result.

\begin{thm}
If $q_1, \cdots, q_n \vdash_{\scriptscriptstyle sqL^{*}} q$, then $q_1, \cdots, q_n\vDash_{\scriptscriptstyle \mathbb{SQW}^{*}} q$.
\end{thm}

\begin{proof}
The proof is similar to Theorem 4.1.
\end{proof}

\noindent\textbf{Acknowledgement}

This study was funded by Shandong Provincial Natural Science Foundation, China
(No. ZR2020MA041).

\raggedright

\begin{bibdiv}
  \begin{biblist}

\bib{1}{article}{
  title={The logic of quasi-MV algebras},
  author={Bou, F{\'e}lix and Paoli, Francesco and Ledda, Antonio and Spinks, Matthew and Giuntini, Roberto},
  journal={Journal of Logic and Computation},
  volume={20},
  number={2},
  pages={619--643},
  year={2010},
  publisher={Oxford University Press}
}

\bib{8}{article}{
  title={The logic of quasi-MV* algebras},
  author={Cai, Lei and Jiang, Yingying and Chen, Wenjuan},
  journal={Submitted}
}

\bib{2}{article}{
  title={Logic with positive and negative truth values},
  author={Chang, Chen Chung},
  year={1971}
}

\bib{2.1}{book}{
  title={Logic and implication},
  author={Cintula, Petr and Noguera, Carles},
  year={2021},
  publisher={Springer}
}

\bib{3}{article}{
  title={A generalization of rotational invariance for complex fuzzy operations},
  author={Dai, Songsong},
  journal={IEEE Transactions on Fuzzy Systems},
  volume={29},
  number={5},
  pages={1152--1159},
  year={2020},
  publisher={IEEE}
}

\bib{4}{article}{
  title={On partial orders in complex fuzzy logic},
  author={Dai, Songsong},
  journal={IEEE Transactions on Fuzzy Systems},
  volume={29},
  number={3},
  pages={698--701},
  year={2019},
  publisher={IEEE}
}

\bib{5}{article}{
  title={On complex fuzzy S-implications},
  author={Dick, Scott},
  journal={IEEE Transactions on Emerging Topics in Computational Intelligence},
  volume={6},
  number={2},
  pages={409--415},
  year={2020},
  publisher={IEEE}
}

\bib{6}{article}{
  title={Toward complex fuzzy logic},
  author={Dick, Scott},
  journal={IEEE Transactions on Fuzzy Systems},
  volume={13},
  number={3},
  pages={405--414},
  year={2005},
  publisher={IEEE}
}

\bib{7}{article}{
  title={On Pythagorean and complex fuzzy set operations},
  author={Dick, Scott and Yager, Ronald R and Yazdanbakhsh, Omolbanin},
  journal={IEEE Transactions on Fuzzy Systems},
  volume={24},
  number={5},
  pages={1009--1021},
  year={2015},
  publisher={IEEE}
}

\bib{7.1}{article}{
  title={Basic fuzzy logic and BL-algebras},
  author={H{\'a}jek, Petr},
  journal={Soft computing},
  volume={2},
  pages={124--128},
  year={1998},
  publisher={Springer}
}

\bib{9}{article}{
  title={Quasi-MV* algebras: a generalization of MV*-algebras},
  author={Jiang, Yingying and Chen, Wenjuan},
  journal={Soft Computing},
  volume={26},
  number={15},
  pages={6999--7015},
  year={2022},
  publisher={Springer}
}

\bib{10}{article}{
  title={MV-algebras and quantum computation},
  author={Ledda, Antonio and Konig, Martinvaldo and Paoli, Francesco and Giuntini, Roberto},
  journal={Studia Logica},
  volume={82},
  pages={245--270},
  year={2006},
  publisher={Springer}
}

\bib{11}{article}{
  title={MV*-algebras},
  author={Lewin, Renato and Sagastume, Marta and Massey, Pedro},
  journal={Logic Journal of the IGPL},
  volume={12},
  number={6},
  pages={461--483},
  year={2004},
  publisher={OUP}
}

\bib{12}{incollection}{
  title={Paraconsistency in Chang's logic with positive and negative truth values},
  author={Lewin, Renato A and Sagastume, Marta S},
  booktitle={Paraconsistency},
  pages={381--396},
  year={2002},
  publisher={CRC Press}
}

\bib{13}{article}{
  title={Chang's $\L^*$ logic},
  author={Lewin, Renato and Sagastume, Marta and Massey, Pedro},
  journal={Logic Journal of the IGPL},
  volume={12},
  number={6},
  pages={485--497},
  year={2004},
  publisher={OUP}
}

\bib{14}{article}{
  title={Comment on Pythagorean and complex fuzzy set operations},
  author={Liu, Lianzhen and Zhang, Xiangyang},
  journal={IEEE Transactions on Fuzzy Systems},
  volume={26},
  number={6},
  pages={3902--3904},
  year={2018},
  publisher={IEEE}
}

\bib{14.1}{article}{
  title={EQ-algebras},
  author={Nov{\'a}k, Vil{\'e}m and De Baets, Bernard},
  journal={Fuzzy sets and systems},
  volume={160},
  number={20},
  pages={2956--2978},
  year={2009},
  publisher={Elsevier}
}

\bib{15}{article}{
  title={Complex fuzzy logic},
  author={Ramot, Daniel and Friedman, Menahem and Langholz, Gideon and Kandel, Abraham},
  journal={IEEE transactions on fuzzy systems},
  volume={11},
  number={4},
  pages={450--461},
  year={2003},
  publisher={IEEE}
}

\bib{15.1}{article}{
  title={Quantum B-algebras},
  author={Rump, Wolfgang},
  journal={Central European Journal of Mathematics},
  volume={11},
  pages={1881--1899},
  year={2013},
  publisher={Springer}
}

\bib{16}{incollection}{
  title={The theory and applications of generalized complex fuzzy propositional logic},
  author={Tamir, Dan E and Last, Mark and Kandel, Abraham},
  booktitle={Soft computing: state of the art theory and novel applications},
  pages={177--192},
  year={2013},
  publisher={Springer}
}



  \end{biblist}
\end{bibdiv}
\end{document}